\documentclass{amsart}
\usepackage{amssymb}
\usepackage{mathrsfs}
\usepackage{upgreek}
\usepackage{hyperref}

\usepackage{tikz}
\usetikzlibrary{cd}

\newcommand{\bk}{\Bbbk}
\newcommand{\C}{\mathbb{C}}
\newcommand{\Z}{\mathbb{Z}}
\newcommand{\bA}{\mathbb{A}}
\newcommand{\Gm}{{\mathbb{G}_{\mathrm{m}}}}
\newcommand{\fa}{\mathfrak{a}}

\newcommand{\sD}{\mathsf{D}}

\newcommand{\cA}{\mathcal{A}}
\newcommand{\cE}{\mathcal{E}}
\newcommand{\cF}{\mathcal{F}}
\newcommand{\cG}{\mathcal{G}}

\newcommand{\cL}{\mathcal{L}}

\newcommand{\sA}{\mathsf{A}}
\newcommand{\sB}{\mathsf{B}}
\newcommand{\sC}{\mathsf{C}}
\newcommand{\sE}{\mathsf{E}}
\newcommand{\sS}{\mathsf{S}}
\newcommand{\sH}{\mathsf{H}}

\newcommand{\Free}{\mathrm{Free}}

\newcommand{\ubk}{\underline{\bk}}
\newcommand{\ubkpi}{\rlap{\underline{\hphantom{$\bk/(\varpi)$}}}\bk/(\varpi){}}
\newcommand{\uubk}{\underline{\underline{\bk}}}
\newcommand{\p}{{}^p}
\newcommand{\Parity}{\mathrm{Parity}}
\newcommand{\Perv}{\mathrm{Perv}}
\newcommand{\scS}{\mathscr{S}}
\newcommand{\pt}{\mathrm{pt}}
\newcommand{\bi}{\mathbf{i}}
\newcommand{\bj}{\mathbf{j}}
\newcommand{\cU}{\mathcal{U}}
\newcommand{\cX}{\mathcal{X}}
\newcommand{\cY}{\mathcal{Y}}
\newcommand{\cZ}{\mathcal{Z}}

\newcommand{\dgmix}{\mathrm{dg}^{\mathrm{mix}}}
\newcommand{\Pargr}[1]{\Parity^\Z(#1,\bk)}
\newcommand{\PargrGm}[1]{\Parity^\Z_\Gm(#1,\bk)}
\newcommand{\DGeq}[1]{\dgmix_\Gm(#1,\bk)}
\newcommand{\DGcon}[1]{\dgmix_\sA(#1,\bk)}
\newcommand{\DGmon}[1]{\dgmix_{\sS,\Theta}(#1,\bk)}
\newcommand{\fDGcon}[1]{\overline{\mathrm{dg}}{}^{\mathrm{mix}}(#1,\bk)}

\newcommand{\D}{\mathbb{D}}
\newcommand{\pari}{{\mathrm{ord}}}
\newcommand{\op}{{\mathrm{op}}}
\newcommand{\bs}{\mathbf{s}}

\newcommand{\mon}{{\mathrm{mon}}}
\newcommand{\For}{\mathrm{For}}
\newcommand{\Mon}{\mathrm{Mon}}
\newcommand{\Coi}{\mathrm{Coi}}
\newcommand{\Inv}{\mathrm{Inv}}
\newcommand{\Nilp}{\mathrm{N}}

\newcommand{\can}{\mathrm{can}}
\newcommand{\var}{\mathrm{var}}
\newcommand{\an}{{\mathrm{an}}}

\newcommand{\bxi}{\bar\xi}
\newcommand{\sr}{\mathsf{r}}
\newcommand{\br}{\bar{\sr}}
\newcommand{\bkappa}{\boldsymbol{\kappa}}

\newcommand{\cRHom}{R\mathcal{H}\mathit{om}}
\newcommand{\cJ}{\mathcal{J}}

\newcommand{\Kb}{K^\mathrm{b}}
\newcommand{\Db}{D^\mathrm{b}}
\newcommand{\Dmix}{D^{\mathrm{mix}}}
\newcommand{\Dbc}{D^{\mathrm{b}}_{\mathrm{c}}}
\newcommand{\Ho}{\mathrm{Ho}}
\newcommand{\coh}{\mathsf{H}}
\newcommand{\fDmix}{\overline{D}{}^{\mathrm{mix}}}

\newcommand{\Hom}{\mathrm{Hom}}
\newcommand{\Ext}{\mathrm{Ext}}
\newcommand{\uHom}{\underline{\Hom}}
\newcommand{\End}{\mathrm{End}}
\newcommand{\uEnd}{\underline{\End}}

\newcommand{\id}{\mathrm{id}}
\newcommand{\la}{\langle}
\newcommand{\ra}{\rangle}
\newcommand{\simto}{\overset{\sim}{\to}}
\newcommand{\sst}{\scriptscriptstyle}
\newcommand{\gen}{\upeta}
\newcommand{\lgmod}{\text{-}\mathrm{gmod}}

\newcommand{\cone}{\mathrm{cone}}
\newcommand{\sneg}{\hspace{-1em}}

\DeclareMathOperator{\tot}{tot}
\DeclareMathOperator{\spn}{span}

\newtheorem{thm}{Theorem}[section]
\newtheorem{lem}[thm]{Lemma}

\newtheorem{prop}[thm]{Proposition}
\newtheorem{hyp}[thm]{Hypothesis}

\theoremstyle{definition}
\newtheorem{defn}[thm]{Definition}

\theoremstyle{remark}
\newtheorem{rmk}[thm]{Remark}
\newtheorem{ex}[thm]{Example}

\numberwithin{equation}{section}

\title[How to glue parity sheaves]{How to glue parity sheaves}

\author{Pramod N. Achar}
 \address{Department of Mathematics\\
   Louisiana State University\\
   Baton Rouge, LA 70803\\
   U.S.A.}
 \email{pramod@math.lsu.edu}

\thanks{The author was partially supported by NSF Grant Nos. DMS-1500890 and DMS-1802241.}

\begin{document}

\begin{abstract}
Let $\cX$ be a stratified space on which the Juteau--Mautner--William\-son theory of parity sheaves is available.  We develop a ``nearby cycles formalism'' in the framework of the homotopy category of parity sheaves on $\cX$, also known as the mixed modular derived category of $\cX$.  This construction is expected to have applications in modular geometric representation theory.
\end{abstract}

\maketitle

\section{Introduction}

\subsection{Overview}

Let $\cX$ be a stratified complex algebraic variety or stack on which the theory of parity sheaves~\cite{jmw} (say, with coefficients in $\bk$) is available.  Following~\cite{ar:mpsfv2}, one can consider the \emph{mixed modular derived category} $\Dmix(\cX,\bk)$, defined to be the bounded homotopy category of chain complexes of parity sheaves.  This category, which is a kind of replacement for the usual bounded derived category of constructible sheaves $\Dbc(\cX,\bk)$, has become a fundamental tool for recent advances in modular geometric representation theory: see~\cite{prinblock, amrw2}.

The main advantage of working with $\Dmix(\cX,\bk)$ rather than $\Dbc(\cX,\bk)$ is that the former is equipped with a notion of ``weights,'' resembling those of mixed $\ell$-adic sheaves or mixed Hodge modules (see~\cite{ar:mpsfv3}).  On the other hand, $\Dmix(\cX,\bk)$ lacks the full sheaf-theoretic machinery available in $\Dbc(\cX,\bk)$: it is difficult to carry out sheaf-theoretic operations on $\Dmix(\cX,\bk)$ unless their classical versions preserve parity sheaves.  For instance, there has so far been no theory of ``nearby cycles'' for $\Dmix(\cX,\bk)$.

The aim of this paper is to develop a nearby cycles formalism for the mixed modular derived category.  More precisely, suppose $\cX$ is equipped with an additional action of $\Gm$, and let $f: \cX \to \bA^1$ be a $\Gm$-equivariant map.  Let $\cX_0 = f^{-1}(0)$, and let $\cX_\gen = f^{-1}(\C \smallsetminus \{0\})$.  We will define a functor
\begin{equation}\label{eqn:nearby-intro}
\Psi_f: \Dmix_\Gm(\cX_\gen) \to \Dmix(\cX_0)
\end{equation}
that in many ways resembles the unipotent part of the classical nearby cycles functor.  In particular, for any $\cF \in \Dmix_\Gm(\cX_\gen,\bk)$, the nearby cycles sheaf $\Psi_f(\cF)$ will be equipped with a natural nilpotent endomorphism
\[
\Nilp: \Psi_f(\cF) \to \Psi_f(\cF)\la 2\ra
\]
that should be thought of as the ``logarithm of the monodromy.'' This construction is expected to have applications to the study of the center of the affine Hecke category: it should make it possible to adapt results of Gaitsgory~\cite{gaitsgory} to the mixed modular setting, and then to write down central objects concretely using the Elias--Williamson calculus~\cite{ew}.  For some examples, see~\cite{arider}.

\subsection{The classical unipotent nearby cycles functor}

Before discussing the ingredients in the definition of~\eqref{eqn:nearby-intro}, let us briefly review the classical situation (for sheaves in the analytic topology).  Following~\cite{ks} or~\cite{reich}, it is given by
\[
\Psi^\an_f = \text{the unipotent part of $\bi^*\bj_*\exp_{X*}\exp_X^*\cF[-1]$},
\]
where the maps $\bi$, $\bj$, and $\exp_X$ are as in the following diagram:
\begin{equation}\label{eqn:nearby-diagram}
\begin{tikzcd}
\cX_0 \ar[r, "\bi"] \ar[d] &
 \cX \ar[d, "f"'] &
 \cX_\gen \ar[l, "\bj"'] \ar[d, "f_\gen"'] &
 \tilde\cX_\gen = \cX_\gen \times_{\C^\times} \C \ar[d] \ar[l, "\exp_X"'] \\
\{0\} \ar[r] &
\C &
\C^\times \ar[l] &
\C \ar[l, "\mathrm{exp}"]
\end{tikzcd}
\end{equation}
An easy computation shows that $\Psi^\an_f(\cF)$ can also be described as the unipotent part of $\bi^*\bj_*\cRHom(f_\gen^*\exp_!\ubk_\C[1],\cF)$.  Note that $\exp_!\ubk_\C$ is the ``regular local system,'' i.e., the (infinite-rank) local system corresponding to the action of $\pi_1(\C^\times,1) \cong \Z$ on its own group algebra $\bk[\Z] = \bk[t,t^{-1}]$.

To get a more concise formula for the unipotent part, one can replace $\exp_!\ubk_\C$ above by the pro-unipotent local system $\cL_\infty$ corresponding to the $\pi_1(\C^\times,1)$-module $\bk[[t-1]]$.  With some additional work, or using the ideas of~\cite{beilinson} (see also~\cite{morel, reich}), one can also replace the $\cRHom$ by a tensor product, and arrive at the formula
\begin{equation}\label{eqn:nearby-analytic}
\Psi_f^\an(\cF) = \bi^*\bj_*(f_\gen^*\cL_\infty(1) \otimes^L \cF).
\end{equation}
(Here, $(1)$ is a Tate twist, introduced to match the conventions of~\cite{morel} and most other sources.  Note that~\cite{beilinson} omits this Tate twist.)  This formula has the advantage of avoiding the nonalgebraic map $\exp_X$, at the expense of explicitly involving the nonconstructible object $\cL_\infty$.

\subsection{Monodromy and constructibility}

The goal of this paper is to adapt~\eqref{eqn:nearby-analytic} to the mixed modular setting.  The main obstacle is that $\cL_\infty$ does not make sense: the framework of parity sheaves in~\cite{jmw} does not allow for infinite-rank stalks or for nonsemisimple local systems.

To solve this problem, we introduce a variant of $\Dmix(\cX,\bk)$ called the \emph{monodromic derived category}, and denoted by  $\Dmix_\mon(\cX,\bk)$.  This category, whose definition is closely inspired by that of ``free-monodromic objects'' in~\cite{amrw1}, \emph{does} allow both nonconstructible objects and nonsemisimple local systems. It fits into a diagram with the $\Gm$-equivariant and ordinary derived categories as shown below:
\begin{equation}\label{eqn:intro-3cat}
\begin{tikzcd}[row sep=tiny]
\Dmix_\Gm(\cX,\bk) \ar[rr, shift left, "\cJ", dashed] \ar[dr, "\For"'] &&
  \Dmix_\mon(\cX,\bk) \ar[ll, shift left, "\Coi"] \\
& \Dmix(\cX,\bk) \ar[ur, hook, "\Mon"']
\end{tikzcd}
\end{equation}
The functors labeled $\For$, $\Mon$, and $\Coi$ are called the \emph{forgetful functor}, the \emph{monodromy functor}, and the \emph{coinvariants of monodromy functor}, respectively.  Among other foundational facts, we will prove that $\Mon$ is full faithful, and that $\Coi$ is left adjoint to $\Mon \circ \For$.  The image of $\Mon$ is precisely the category of constructible objects in $\Dmix_\mon(\cX,\bk)$.

The functor $\cJ$ will be defined in Section~\ref{sec:jordan} under an additional assumption on $\cX$, called \emph{$R$-triviality}.  (In the context of~\eqref{eqn:nearby-diagram}, $\cX_\gen$ satisfies this assumption, but $\cX_0$ does not.)  This functor is equpped with a natural isomorphism
\begin{equation}\label{eqn:intro-coi-j}
\Coi(\cJ(\cF)) \cong \cF
\end{equation}
and a nonzero natural transformation (not an isomorphism)
\begin{equation}\label{eqn:intro-j-cov}
\cJ(\cF) \to \Mon(\For(\cF)).
\end{equation}
Suppose now that $\cF$ is a perverse sheaf.  Informally,~\eqref{eqn:intro-coi-j} says that $\cJ(\cF)$ is ``acyclic'' for the coinvariants of monodromy functor---in other words, the output of $\cJ(\cF)$ behaves like a projective (or pro-unipotent) object with respect to the monodromy action. Then,~\eqref{eqn:intro-j-cov} lets us identify it as the pro-unipotent cover of $\Mon(\For(\cF))$.

The key idea in this paper is that the functor $\cJ$ can serve as a replacement for the expression $f_\gen^*\cL_\infty \otimes^L({-})$ from~\eqref{eqn:nearby-analytic}.  The actual definition of the mixed modular nearby cycles functor~\eqref{eqn:nearby-intro} is
\[
\Psi_f(\cF) = \Mon^{-1}(\bi^*\bj_*\cJ(\cF))\la -2\ra.
\]
Of course, for $\Mon^{-1}$ to make sense here, we will need to prove that $\bi^*\bj_*\cJ(\cF)$ is constructible  (even though $\cJ(\cF)$ is not).  This will be a consequence of a more general constructibility theorem proved in Section~\ref{sec:rfree-constr}.  The same issue arises in the analytic case: a fundamental theorem about $\Psi^\an_f$ is that it preserves constructibility, even though~\eqref{eqn:nearby-analytic} involves the nonconstructible object $\cL_\infty$. 

\subsection{\texorpdfstring{$t$}{t}-exactness}

The second fundamental result about the classical nearby cycles functor $\Psi^\an_f$ is that it is $t$-exact for the perverse $t$-structure.  In this paper, following~\cite{reich}, we explain how to deduce the $t$-exactness of $\Psi_f$ from the assumption that $\bj_!$ and $\bj_*$ are $t$-exact (see Hypothesis~\ref{hyp:j-exact}). The inclusion map $\bj: \cX_\gen \hookrightarrow \cX$ is an affine morphism, so in the classical setting, the $t$-exactness of these functors is a theorem.  In the mixed modular case, this is assumption is related to Assumption (A2) from~\cite[\S3.2]{ar:mpsfv2}.  The latter has been checked in the important case of (Kac--Moody) flag varieties.  It seems likely that for applications in geometric representation theory, it will be possible to verify Hypothesis~\ref{hyp:j-exact} in the relevant cases.  (See~\cite{arider} for some such cases.)  Unfortunately, Hypothesis~\ref{hyp:j-exact}, and hence the $t$-exactness of $\Psi_f$, remains conjectural in general.

\subsection{Contents of the paper}

Section~\ref{sec:prelim} establishes notation and conventions for graded rings and for parity sheaves.  In Sections~\ref{sec:3derived}--\ref{sec:monodromy}, we define and establish basic properties of the categories and functors in~\eqref{eqn:intro-3cat}.

Sections~\ref{sec:recollement} and~\ref{sec:perverse} deal with the recollement formalism and the perverse $t$-structure for the various categories in~\eqref{eqn:intro-3cat}.  (In some important special cases, these results were previously obtained in~\cite{ar:mpsfv2}. See also~\cite{arv:mps} for related results.)  
This part of the paper is needed for the functors $\bi^*$ and $\bj_*$ to make sense.

In Section~\ref{sec:rfree-constr}, we define the notion of constructibility for objects in $\Dmix_\mon(\cX,\bk)$.  The main result of this section states that when $\cX$ satisfies an additional technical condition (called \emph{$R$-freeness}), every object $\Dmix_\mon(\cX,\bk)$ is constructible.

Section~\ref{sec:jordan} contains the definition and basic properties of the functor $\cJ$.  The heart of the paper is Section~\ref{sec:nearby}, which defines and proves the basic properties of the nearby cycles functor $\Psi_f$, as well as two related functors, called the maximal extension functor $\Xi_f$ and the vanishing cycles functor $\Phi_f$.  We conclude the paper in Section~\ref{sec:examples} with a few examples.

\subsection{Acknowledgments}

The ideas in this paper have been strongly influenced by conversations with Shotaro Makisumi, Simon Riche, and Laura Rider.  Thanks also to Ben Elias and Geordie Williamson for helpful remarks at an early stage of this work.

\section{Preliminaries}
\label{sec:prelim}

\subsection{Bigraded rings and modules}
\label{ss:bigraded}

Let $\bk$ be a field or a complete local principal ideal domain.  Given a bigraded $\bk$-module $M = \bigoplus_{i,j \in \Z} M^i_j$, we let $M[n]$ and $M\la k\ra$ be the bigraded modules given by
\[
(M[n])^i_j = M^{i+n}_j
\qquad\text{and}\qquad
(M\la n\ra)^i_j = M^i_{j-n}.
\]
We also define an operation $M \mapsto M\{n\}$ by $M\{n\} = M\la -n\ra[n]$.  If $m \in M^i_j$, we say that $m$ is homogeneous of bidegree $\binom{i}{j}$.  We call $i$ the \emph{cohomological degree} of $m$, and we write
\[
|m| = i.
\]
If $M$ and $N$ are bigraded $\bk$-modules, we define bigraded $\bk$-modules $M \otimes N$ and $\uHom(M,N)$ in the usual way:
\[
(M \otimes N)^i_j =
\bigoplus_{\substack{p+q=i\\ r+s=j}} M^p_r \otimes N^q_s,
\qquad
\uHom(M,N)^i_j =
\bigoplus_{\substack{q-p=i\\s-r=j}} \Hom(M^p_r, N^q_s).
\]

Let $\xi$, $\bxi$, $\sr$, and $\br$ be four indeterminates with bidgrees
\[
\textstyle
\deg \xi = \binom{2}{2}, \qquad
\deg \bxi = \binom{1}{2}, \qquad
\deg \sr = \binom{0}{-2}, \qquad
\deg \br = \binom{-1}{-2}.
\]
We define various bigraded symmetric algebras on these generators:
\begin{equation}\label{eqn:rings-defn}
\begin{gathered}
\begin{aligned}
R &= \bk[\xi] &\qquad&\qquad& R^\vee &= \bk[\sr] \\
\Lambda &= \bk[\bxi] &&& \Lambda^\vee &= \bk[\br] \\
\sA &= \bk[\xi,\bxi] = \Lambda \otimes R &&& \sA^\vee &= \bk[\sr,\br] = \Lambda^\vee \otimes R^\vee
\end{aligned}
\\
\sS = \bk[\xi,\sr] = R^\vee \otimes R.
\end{gathered}
\end{equation}
These symmetric algebras are to be understood in the graded sense, as in~\cite[\S 3.1]{amrw1}.  Since $\bxi$ and $\br$ have odd cohomological degree, the rings $\Lambda$ and $\Lambda^\vee$ are exterior algebras on one generator.  The rings $R$, $R^\vee$, and $\sS$ are polynomial rings.  The element
\[
\Theta = \sr\xi \in \sS^2_0
\]
will play an important role in the sequel.

Each of the rings defined above has a unit map, denoted by $\iota_R: \bk \to R$, $\iota_\sA: \bk \to \sA$, etc., and a counit map, denoted by $\varepsilon_R: R \to \bk$, $\varepsilon_\sA: \sA \to \bk$, etc.

Equip $\sA$ and $\sA^\vee$ with differentials $\kappa$ and $\kappa^\vee$ given by setting
\begin{align*}
\kappa(\xi) &= 0, & \kappa^\vee(\sr) &= 0, \\
\kappa(\bxi) &= \xi, & \kappa^\vee(\br) &= \sr,
\end{align*}
and then extending by the Leibniz rule.  The following fact is well known.

\begin{lem}\label{lem:sa-counit}
The maps $\varepsilon_\sA: (\sA,\kappa) \to \bk$ and $\varepsilon_{\sA^\vee}: (\sA^\vee,\kappa^\vee) \to \bk$ are quasi-isomorphisms.
\end{lem}

Let $\bk\la\br,\bxi\ra$ be the free associative bigraded $\bk$-algebra on the generators $\br$ and $\bxi$, and then let
\[
\sE = \bk\la \br,\bxi\ra /(\br^2 = 0, \bxi^2 =0, \br\bxi + \bxi\br = 1).
\]
This ring contains $\Lambda^\vee$ and $\Lambda$ as subrings, and the multiplication map
\begin{equation}\label{eqn:ering-isom}
\Lambda^\vee \otimes \Lambda \to \sE
\end{equation}
is an isomorphism of $\bk$-modules (but not a ring isomorphism).  There is an isomorphism of bigraded $\bk$-algebras
\begin{equation}\label{eqn:ering-end}
\sE \simto \uEnd(\Lambda^\vee)
\end{equation}
given by letting the subring $\Lambda^\vee \subset \sE$ act on $\Lambda^\vee$ by multiplication, and letting $\Lambda \subset \sE$ act by contraction (i.e., $\bxi$ gives the endomorphism $1 \mapsto 0$, $\br \mapsto 1$ of $\Lambda^\vee$).

Finally, let
\[
\sB = R^\vee \otimes \sE \otimes R.
\]
All seven rings defined in~\eqref{eqn:rings-defn} can be regarded as subrings of $\sB$.  (Unlike those rings, however, $\sB$ is \emph{not} graded-commutative.)  Let
\[
\omega = \sr\bxi + \br\xi \in \sB^1_0.
\]
It is easy to see that
\begin{equation}\label{eqn:omega}
\omega^2 = \Theta.
\end{equation}
Equip $\sB$ with the differential $\bkappa$ given by
\[
\bkappa(b) = \omega b + (-1)^{|b|+1}b \omega.
\]

\begin{lem}\label{lem:ab-qiso}
The inclusion map $(\sA,\kappa) \to (\sB,\bkappa)$ is a quasi-isomorphism.
\end{lem}
\begin{proof}
In $\sB$, we have $\bkappa(\xi) = \omega\xi - \xi\omega = 0$, and $\bkappa(\bxi) = \br\xi\bxi + \bxi\br\xi = \xi$.  In other words, the inclusion map $\sA \to \sB$ is at least a chain map.  Similarly, the inclusion map $(\sA^\vee,\kappa^\vee) \to (\sB,\bkappa)$ is also a chain map.  Since $\bkappa$ obeys the Leibniz rule, the multiplication map
\[
\sA^\vee \otimes \sA \to \sB
\]
is again a chain map (but no longer a ring homomorphism).  In view of~\eqref{eqn:ering-isom}, this map is an isomorphism of $\bk$-modules, and hence of chain complexes.  Since $\sA$ and $\coh^\bullet(\sA)$ are both flat over $\bk$, we deduce that $\coh^\bullet(\sB)  \cong \coh^\bullet(\sA^\vee) \otimes \coh^\bullet(\sA) \cong \bk$, as desired.
\end{proof}

\subsection{Parity sheaves}
\label{ss:parity}

Let $H$ be an algebraic group over $\C$.  Suppose that $H$ admits an action of $\Gm$ by group automorphisms, so that we may form the group $\Gm \ltimes H$.

Let $X$ be a variety over $\C$ with an action of $\Gm \ltimes H$.  Assume that $X$ is equipped with an algebraic stratification $(X_s)_{s \in \scS}$ such that:
\begin{enumerate}
\item Each stratum $X_s$ is preserved by the action of $\Gm \ltimes H$.
\item The equivariant cohomology $\coh^\bullet_{\Gm \ltimes H}(X_s,\bk)$ of each stratum is concentrated in even degrees and free over $\bk$.
\item Every $\Gm \ltimes H$-equivariant local system on every stratum $X_s$ is trivial.
\end{enumerate}
Recall that $\bk$ is a field or a complete local principal ideal domain.  Under the assumptions above, it makes sense to speak of $\Gm \ltimes H$-equivariant parity sheaves on $X$, and~\cite[Theorem~2.12]{jmw} says that there is at most one indecomposable parity sheaf (up to shift and isomorphism) supported on each stratum closure.  We add one more assumption:
\begin{enumerate}
\setcounter{enumi}{3}
\item For each stratum $X_s$, there exists an indecomposable $\Gm \ltimes H$-equivariant parity sheaf $\cE_s$ supported on $\overline{X_s}$ and satisfying $\cE_s|_{X_s} \cong \ubk[\dim X_s]$.
\end{enumerate}
The action of the group $H$ is relevant for applications, so it is important to make sure our set-up keeps track of $H$-equivariance.  On the other hand, the $H$-action plays no role in the present paper.  For brevity, it is convenient to suppress $H$ from the notation.  To this end, we will use the ``stacky'' notation
\[
\cX := X/H
\qquad\text{and}\qquad
\cX_s := X_s/H\quad\text{for each $s \in \scS$.}
\]
We henceforth denote the category of $\Gm \ltimes H$-equivariant parity sheaves on $X$ by
\[
\Parity_\Gm(\cX,\bk).
\]
This category inherits a ``cohomological shift'' functor from the derived category $\Db_\Gm(\cX,\bk)$. Following~\cite{ar:mpsfv2}, we denote this functor by $\{1\}$.  For any two objects $\cF, \cG \in \Parity_\Gm(\cX,\bk)$, the graded $\bk$-module $\bigoplus_n \Hom(\cF,\cG\{n\})$ naturally has the structure of a graded module over $\coh^\bullet_\Gm(\pt,\bk)$.

We make these spaces into \emph{bigraded} modules that are ``concentrated on the diagonal'' as follows: for $\cF, \cG \in \Parity_\Gm(\cX,\bk)$, let $\uHom(\cF,\cG)$ be the bigraded $\bk$-module given by
\[
\uHom(\cF,\cG)^i_j = \begin{cases}
\Hom(\cF,\cG\{i\}) & \text{if $i = j$,} \\
0 & \text{otherwise.}
\end{cases}
\]
Next, recall that $\coh^\bullet_\Gm(\pt,\bk)$ is the symmetric algebra on $\coh^2_\Gm(\pt,\bk) = \bk \otimes_\Z X(\Gm)$, where $X(\Gm)$ is the character lattice of $\Gm$.  Identify the canonical generator of $X(\Gm)$ with the indeterminate $\xi$.  In this way, we obtain an identification
\[
R = \bk[\xi] = \coh^\bullet_\Gm(\pt,\bk).
\]
For $\cF, \cG \in \Parity_\Gm(\cX,\bk)$, the space $\uHom(\cF,\cG)$ is then a bigraded $R$-module.

Note that \emph{we have not imposed any assumptions} on the structure of $\uHom(\cF,\cG)$ as an $R$-module.  However, it will sometimes be useful to consider the following special cases.

\begin{defn}
The space $\cX$ is said to be \emph{$R$-free} if for every stratum $\cX_s$, 
the cohomology
$\coh^\bullet_\Gm(\cX_s,\bk)$ is free as an $R$-module.  By (the proof of)~\cite[Proposition~2.6]{jmw}, it follows that for any two objects $\cF, \cG \in \Parity_\Gm(\cX,\bk)$, the space $\uHom(\cF,\cG)$ is a free $R$-module.

On the other hand, $\cX$ is said to be \emph{$R$-trivial} if for any two parity sheaves $\cF, \cG \in \Parity_\Gm(\cX,\bk)$, the element $\xi \in R$ acts by $0$ on $\uHom(\cF,\cG)$.  If $\cX$ is $R$-trivial, then every locally closed union of strata in $\cX$ is also $R$-trivial.
\end{defn}

\begin{rmk}\label{rmk:rfree-constr}
If $\cX$ is $R$-free, then the $H$-equivariant cohomology (forgetting the $\Gm$-action) of a stratum is given by
\[
\coh^\bullet_H(X_s,\bk) = \coh^\bullet(\cX_s,\bk) \cong \bk \otimes_R \coh^\bullet_\Gm(\cX_s,\bk).
\]
In particular, the $H$-equivariant cohomology of each stratum is again even, so it makes sense to drop the $\Gm$-equivariance and consider the category $\Parity(\cX,\bk)$ of $H$-equivariant parity sheaves on $X$.

However, if $\cX$ is not $R$-free, the cohomology $\coh^\bullet_H(X_s,\bk)$ can fail to be even, so $\Parity(\cX,\bk)$ does not make sense in general.
\end{rmk}

\subsection{Graded parity sheaves}

We define a \emph{graded parity sheaf} to simply be a formal expression of the form
\[
\cF = \bigoplus_{i \in \Z} \cF^i[-i],
\]
where $\cF^i \in \Parity_\Gm(\cX,\bk)$, and where all but finitely many terms are zero.  (Of course, any parity sheaf can be regarded as a graded parity sheaf.) If $\cF$ and $\cG$ are graded parity sheaves, we define
\[
\uHom(\cF,\cG) = \bigoplus_{i,j \in \Z} \uHom(\cF^i,\cG^j)[i-j].
\]
The notion of a graded parity sheaf is equivalent to the notion of a ``parity sequence'' from~\cite{amrw1}.  The category of graded parity sheaves is denoted by $\PargrGm{\cX}$.  For $\cF \in \PargrGm{\cX}$, define $\cF[n]$ in the obvious way, and define $\cF\la n\ra$ by $\cF\la n\ra = \cF\{-n\}[n]$.

\section{Three derived categories}
\label{sec:3derived}

In this section, we will define the three categories in~\eqref{eqn:intro-3cat}, as well as the forgetful functor.  Each category arises as the homotopy category of a $\bk$-linear dgg (differential bigraded) category.  In all three cases, the objects are graded parity sheaves with some additional data.

\subsection{The \texorpdfstring{$\Gm$}{Gm}-equivariant derived category}
\label{ss:dmix-gm}

Let $\DGeq{\cX}$ be the dgg category defined as follows:
\begin{itemize}
\item The objects are pairs $(\cF,\delta)$, where $\cF \in \PargrGm{\cX}$, and where
\[
\delta \in \uEnd(\cF)^1_0
\qquad\text{satisfies}\qquad
\delta \circ \delta = 0.
\]
\item Given two objects $(\cF,\delta_\cF)$ and $(\cG,\delta_\cG)$ as above, the morphism space is $\uHom(\cF,\cG)$, made into a chain complex with the differential
\[
d(f) = \delta_\cG \circ f + (-1)^{|f|+1}f \circ \delta_\cF.
\]
\end{itemize}
We then set
\[
\Dmix_\Gm(\cX,\bk) = \Ho(\DGeq{\cX}),
\]
and we call this the \emph{$\Gm$-equivariant derived category} of $\cX$.  An object of $\DGeq{\cX}$ can be thought of as just a chain complex over the additive category $\Parity_\Gm(\cX,\bk)$.  In other words, $\Dmix_\Gm(\cX,\bk)$ can be identified with $\Kb\Parity_\Gm(\cX,\bk)$.

Note that any parity sheaf $\cF$ can be regarded as an object of $\Dmix_\Gm(\cX,\bk)$ by equipping it with the zero differential. 

\begin{ex}\label{ex:gm-der}
Let $\cX = \bA^1$, stratified as the union of $\cX_0 = \{0\}$ and $\cX_1 = \bA^1 \smallsetminus \{0\}$, and equipped with the standard action of $\Gm$.  Then $\cE_0 = \ubk_{\cX_0}$, and $\cE_1 = \ubk_{\overline{\cX_1}}\{1\}$.  There is a restriction map $\epsilon: \cE_1 \to \cE_0\{1\} = \cE_0\la -1\ra[1]$.  This map can be taken to be the differential of an object of $\Dmix_\Gm(\cX,\bk)$ with underlying graded parity sheaf $\cE_1 \oplus \cE_0\la -1\ra$.  We draw this object as follows:
\[
\begin{tikzcd}[column sep=large]
\cE_1 \ar[r, "{\sst[1]}" description, "\epsilon" above=2] & \cE_0\la -1\ra.
\end{tikzcd}
\]
\end{ex}

\subsection{The constructible derived category}
\label{ss:dgmix-con}

Recall the dgg ring $\sA$ and its differential $\kappa$.  Given two graded parity sheaves $\cF$ and $\cG$, we also write $\kappa$ for the differential on the $\bk$-module
\[
\sA \otimes_R \uHom(\cF,\cG)
\]
given by $\kappa(a \otimes f) = \kappa(a) \otimes f$.

Let $\DGcon{\cX}$ be the dgg category defined as follows:
\begin{itemize}
\item The objects are pairs $(\cF,\delta)$, where $\cF \in \PargrGm{\cX}$, and where
\[
\delta \in (\sA \otimes_R \uEnd(\cF))^1_0
\qquad\text{satisfies}\qquad
\delta \circ \delta = 0.
\]
\item Given two objects $(\cF,\delta_\cF)$ and $(\cG,\delta_\cG)$ as above, the morphism space is $\sA \otimes_R \uHom(\cF,\cG) \cong \Lambda \otimes \uHom(\cF,\cG)$, made into a chain complex with the differential
\[
d(f) = \delta_\cG \circ f + (-1)^{|f|+1}f \circ \delta_\cF + \kappa(f).
\]
\end{itemize}
We then set
\[
\Dmix(\cX,\bk) = \Ho(\DGcon{\cX}),
\]
and we call this the \emph{constructible derived category} of $\cX$.

Note that any parity sheaf $\cF$ can be regarded as an object of $\Dmix(\cX,\bk)$ by equipping it with the zero differential.

\begin{ex}\label{ex:con-der}
Let $\cX$, $\cE_0$, and $\cE_1$ be as in Example~\ref{ex:gm-der}. There is a map $\eta: \cE_0\{-1\} \to \cE_1$ that is Verdier dual to $\epsilon: \cE_1 \to \cE_0\{1\}$.  It can be shown that $\epsilon \circ \eta = \xi\cdot \id_{\cE_0\{-1\}}$.  There is an object of $\Dmix(\cX,\bk)$ given by
\[
\begin{tikzcd}[column sep=large]
\cE_0\la 1\ra \ar[r, "{\sst[1]}" description, "\eta" below=2]
  \ar[rr, bend left=20, "{\sst[1]}" description, "-\bxi\cdot\id" above=2] &
\cE_1 \ar[r, "{\sst[1]}" description, "\epsilon" below=2] &
\cE_0\la -1\ra.
\end{tikzcd}
\]
\end{ex}

\begin{rmk}\label{rmk:hom-nilp}
Degree considerations show that for any two graded parity sheaves $\cF$ and $\cG$, both $\uHom(\cF,\cG)^0_j$ and $(\sA \otimes_R \uHom(\cF,\cG))^0_j$ can be nonzero for only finitely many $j$.  As a consequence, if $\cF$ and $\cG$ are objects of either $\Dmix_\Gm(\cX,\bk)$ or $\Dmix(\cX,\bk)$, the direct sum
\[
\bigoplus_{n \in \Z} \Hom(\cF,\cG\la n\ra)
\]
has only finitely many nonzero terms.
\end{rmk}

\subsection{The forgetful functor}

For any $\cF, \cG \in \PargrGm{\cX}$, the inclusion map $R \hookrightarrow \sA$ induces a map of $\uHom$-spaces that we denote by
\[
\For: \uHom(\cF,\cG) \to \sA \otimes_R \uHom(\cF,\cG).
\]
If $(\cF, \delta)$ is an object of $\DGeq{\cX}$, then it is easy to see that $(\cF, \For(\delta))$ is an object of $\DGcon{\cX}$.  We therefore obtain a functor denoted by $\For: \DGeq{\cX} \to \DGcon{\cX}$.  After passing to homotopy categories, we obtain a functor
\[
\For: \Dmix_\Gm(\cX,\bk) \to \Dmix(\cX,\bk),
\]
called the \emph{forgetful functor}.

\subsection{The constructible derived category in the \texorpdfstring{$R$}{R}-free case}
\label{ss:rfree-constr}

At first glance, the definition of the constructible derived category $\Dmix(\cX,\bk)$ in~\S\ref{ss:dgmix-con} does not resemble that in~\cite{ar:mpsfv2}.  Let us explain how to compare the two.  Assume that $\cX$ is $R$-free, and recall from Remark~\ref{rmk:rfree-constr} that we may drop the $\Gm$-equivariance, and consider the category $\Parity(\cX,\bk)$ of $H$-equivariant parity sheaves.  We can also form its graded version $\Pargr{\cX}$.  Let $\fDGcon{\cX}$ be the dgg category whose objects are pairs $(\cF,\delta)$, with $\cF \in \Pargr{\cX}$, and with $\delta \in \uEnd(\cF)^1_0$ satisfying $\delta \circ \delta = 0$.  Let
\[
\fDmix(\cX,\bk) = \Ho(\fDGcon{\cX}).
\]
Then $\fDmix(\cX,\bk)$ can be identified with $\Kb\Parity(\cX,\bk)$. The category $\fDmix(\cX,\bk)$ is precisely what was called the \emph{mixed derived category} in~\cite{ar:mpsfv2}.  We will see below that $\fDmix(\cX,\bk)$ and $\Dmix(\cX,\bk)$ are equivalent.

The functor $\PargrGm{\cX} \to \Pargr{\cX}$ that forgets the $\Gm$-equivariance will be denoted by $\cF \mapsto \bar \cF$.  By~\cite[Lemma~A.11]{ar:agsr}, for $\cF,\cG \in \PargrGm{\cX}$, there is a natural isomorphism
\[
\bk \otimes_R \uHom_{\PargrGm{\cX}}(\cF,\cG) \simto \uHom_{\Pargr{\cX}}(\bar\cF, \bar\cG).
\]
The counit $\varepsilon_\sA: \sA \to \bk$ then induces a map
\begin{equation}\label{eqn:rfree-uhom}
\sA \otimes_R \uHom_{\PargrGm{\cX}}(\cF,\cG) \to \uHom_{\Pargr{\cX}}(\bar\cF,\bar\cG).
\end{equation}
Because $\uHom(\cF,\cG)$ is free (and hence flat) over $R$, Lemma~\ref{lem:sa-counit} implies that this is a quasi-isomorphism.

Now let $(\cF,\delta) \in \DGcon{\cX}$.  Let $\bar\delta \in \uEnd(\bar\cF)$ be the image of $\delta$ under~\eqref{eqn:rfree-uhom}.  This element satisfies $\bar\delta^2 = 0$, so we get a functor
\[
F: \DGcon{\cX} \to \fDGcon{\cX}.
\]
Passing to homotopy categories, we obtain a functor
\[
F: \Dmix(\cX,\bk) \to \fDmix(\cX,\bk)
\]

\begin{lem}\label{lem:dmix-classical}
Assume that $\cX$ is $R$-free.  The functor $F: \Dmix(\cX,\bk) \to \fDmix(\cX,\bk)$ is an equivalence of categories.
\end{lem}

\begin{proof}
The fact that~\eqref{eqn:rfree-uhom} is a quasi-isomorphism implies that $F$ is fully faithful. The category $\fDmix(\cX,\bk) \cong \Kb\Parity(\cX,\bk)$ is generated as a triangulated category by parity sheaves (with zero differential).  Since these objects are clearly in the image of $F$, and since $F$ is a triangulated functor (see Section~\ref{sec:triangulated} for the triangulated structure on $\Dmix(\cX,\bk)$), we conclude that $F$ is essentially surjective as well.
\end{proof}

Thus, for an $R$-free variety, the notation $\Dmix(\cX,\bk)$ as used in the present paper is consistent in spirit with that of~\cite{ar:mpsfv2}.

On the other hand, if $\cX$ is not $R$-free, it cannot be studied in the framework of~\cite{ar:mpsfv2} (because $\Parity(\cX,\bk)$ does not make sense).  The definition of $\Dmix(\cX,\bk)$ in the present paper is new in that case.

\subsection{The monodromic derived category}

Recall the element $\Theta \in \sS^2_0$.  Let $\DGmon{\cX}$ be the dgg category defined as follows:
\begin{itemize}
\item The objects are pairs $(\cF,\delta)$, where $\cF \in \PargrGm{\cX}$, and where
\[
\delta \in (\sS \otimes_R \uEnd(\cF))^1_0
\qquad\text{satisfies}\qquad
\delta \circ \delta = \Theta \cdot \id_\cF.
\]
\item Given two objects $(\cF,\delta_\cF)$ and $(\cG,\delta_\cG)$ as above, the morphism space is $\sS \otimes_R \uHom(\cF,\cG) \cong R^\vee \otimes \uHom(\cF,\cG)$, made into a chain complex with the differential
\[
d(f) = \delta_\cG \circ f + (-1)^{|f|+1}f \circ \delta_\cF.
\]
\end{itemize}
(The fact that $d \circ d = 0$ follows from the fact that $f \circ \Theta\id_\cF = \Theta\id_\cG \circ f = \Theta f$.) We then set
\[
\Dmix_\mon(\cX,\bk) = \Ho(\DGmon{\cX}),
\]
and we call it the \emph{monodromic derived category}.

The ring $R^\vee$ can be regarded as a subring of $\sS \otimes_R \uEnd(\cF) \cong R^\vee \otimes \uEnd(\cF)$ for any object $\cF \in \DGmon{\cX}$.  Moreover, the definition of the differential shows that $d(\sr) = 0$, so we obtain a morphism $\sr \cdot \id_\cF: \cF \to \cF\la 2\ra$.

\begin{defn}\label{defn:monodromy}
For $\cF \in \Dmix_\mon(\cX,\bk)$, the map
\[
\Nilp_\cF = \sr\cdot \id_\cF: \cF \to \cF\la 2\ra
\]
is called the \emph{monodromy endomorphism}.
\end{defn}

It is immediate from the definitions that the monodromy endomorphism commutes with all morphisms in $\Dmix_\mon(\cX,\bk)$.  That is, for any morphism $f: \cF \to \cG$, we have $\Nilp_\cG \circ f = f\la 2\ra \circ \Nilp_\cF$.  (Of course, it is a slight misnomer to call it an ``endomorphism.'')  

It is a bit trickier to exhibit examples of objects in $\Dmix_\mon(\cX,\bk)$ than in $\Dmix_\Gm(\cX,\bk)$ or $\Dmix(\cX,\bk)$, since the zero map is not a valid differential for a graded parity sheaf $\cF$ unless $\Theta\cdot \id_\cF = 0$.  Nevertheless, given any parity sheaf $\cF$, one can construct an object of $\Dmix_\mon(\cX,\bk)$ from it as follows:
\[
\begin{tikzcd}[column sep=large]
\cF \ar[r, shift left=1.5, "{\sst[1]}" description, "\xi\cdot \id" above=2] &
\cF\la -2\ra[1] \ar[l, shift left=1.5, "{\sst[1]}" description, "\sr\cdot \id" below=2]
\end{tikzcd}
\]

\begin{ex}\label{ex:mon-der1}
In the context of Examples~\ref{ex:gm-der} and~\ref{ex:con-der}, we can write down some more complicated objects in $\Dmix_\mon(\cX,\bk)$.  Here is an object that is related to Example~\ref{ex:gm-der}:
\[
\begin{tikzcd}[column sep=large]
\cE_1 \ar[r, shift left=1.5, "{\sst[1]}" description, "\epsilon" above=2] &
\cE_0\la -1\ra \ar[l, shift left=1.5, "{\sst[1]}" description, "\sr\eta" below=2]
\end{tikzcd}
\]
Let $\cG$ denote this object. It can be shown that $\bigoplus_{n \in \Z} \Hom(\cG,\cG\la n\ra) \cong R^\vee$.  In particular, this object shows that Remark~\ref{rmk:hom-nilp} need not hold in $\Dmix_\mon(\cX,\bk)$.
\end{ex}

\begin{ex}\label{ex:mon-der2}
The following object is related to Example~\ref{ex:con-der}:
\[
\begin{tikzcd}[column sep=large, row sep=large]
\cE_0\la 1\ra \ar[r, "{\sst[1]}" description, "\eta" above=2]
  \ar[d, shift right=2, "{\sst[1]}" description, "\xi\cdot\id" left=2] &
\cE_1 \ar[r, "{\sst[1]}" description, "\epsilon" above=2]
  \ar[d, shift right=2, "{\sst[1]}"{description, pos=0.25}, "\xi\cdot\id"{pos=0.25,left=2}] &
\cE_0\la -1\ra
  \ar[d, shift right=2, "{\sst[1]}" description, "\xi\cdot\id" left=2]
\\
\cE_0\la -1\ra[1] \ar[r, "{\sst[1]}" description, "-\eta" below=2]
  \ar[u, shift right=2, "{\sst[1]}" description, "\sr\cdot\id" right=2]
  \ar[urr, "{\sst[1]}"{description, pos=0.15}, "-\id"{pos=0.15, above=2}] &
\cE_1\la -2\ra[1] \ar[r, "{\sst[1]}" description, "-\epsilon" below=2]
  \ar[u, shift right=2, "{\sst[1]}"{description, pos=0.25}, "\sr\cdot\id"{pos=0.25,right=2}] &
\cE_0\la -3\ra[1]
  \ar[u, shift right=2, "{\sst[1]}" description, "\sr\cdot\id" right=2]
\end{tikzcd}
\]
See Remark~\ref{rmk:mon-concrex} for a discussion of this object.
\end{ex}

\section{Triangulated structure}
\label{sec:triangulated}

For $\cF \in \PargrGm{\cX}$, we of course have $\uHom(\cF,\cF[1])^{-1}_0 \cong \uEnd(\cF)^0_0$.  Let
\[
s_\cF \in \uHom(\cF,\cF[1])^{-1}_0
\]
denote the element corresponding to the identity map $\id_\cF \in \uEnd(\cF)^0_0$.  In a minor abuse of notation, we also write $s$ to mean the element $1 \otimes s$ in $\sA \otimes_R \uHom(\cF,\cF[1])$ or $\sS \otimes_R \uHom(\cF,\cF[1])$.

There are natural isomorphisms
\begin{align*}
\bs &: \uHom(\cF,\cG) \simto \uHom(\cF[1],\cG[1]), \\
\bs &: \sA \otimes_R \uHom(\cF,\cG) \simto \sA \otimes_R \uHom(\cF[1],\cG[1]), \\
\bs &: \sS \otimes_R \uHom(\cF,\cG) \simto \sS \otimes_R \uHom(\cF[1],\cG[1])
\end{align*}
given by $\bs(f) = s \circ f \circ s^{-1}$.  Note that in the case of $\sA \otimes_R \uHom(\cF,\cG)$, additional signs may arise because of the ``Koszul sign rule'': for $a \in \sA$ and $f \in \uHom(\cF,\cG)$, we have
\begin{multline*}
\bs(a \otimes f) = (\id \otimes s) \circ (a \otimes f)  \circ(\id \otimes s^{-1}) \\
= (-1)^{|s||a| + |\id||s| + |\id||f|}((\id \circ a \circ \id) \otimes (s \circ f \circ s^{-1}))
= (-1)^{|a|}(a \otimes \bs(f)).
\end{multline*}
(In principle, the same phenomenon occurs in $\sS \otimes_R \uHom(\cF,\cG)$, but of course all elements of $\sS$ have even cohomological degree, so the signs are invisible.)  Concretely, if we write $f \in (\sA \otimes_R \uHom(\cF,\cG))^i_j$ as $f = f_0 + \bxi f_1$ with $f_0 \in \uHom(\cF,\cG)^i_j$ and $f_1 \in \uHom(\cF,\cG)^{i-1}_{j-2}$, then
\[
\bs(f) = \bs(f_0 + \bxi f_1) = \bs(f_0) - \bxi \bs(f_1).
\]

We are now ready to define shift functors and mapping cones in the various categories from~\S\ref{sec:3derived}.  Let $\tilde\sD$ denote one of $\DGeq{\cX}$, $\DGcon{\cX}$, or $\DGmon{\cX}$.  Define a functor $[1]: \tilde\sD \to \tilde\sD$ on objects and on morphisms by
\[
(\cF,\delta)[1] = (\cF[1], -\bs(\delta))
\qquad\text{and}\qquad
f[1] = \bs(f).
\]
(Here, $f$ is an element of $\uHom(\cF,\cG)$ or $\sA \otimes_R \uHom(\cF,\cG)$ or $\sS \otimes_R \uHom(\cF,\cG)$, as appropriate.)  

Next, let $f: (\cF,\delta_\cF) \to (\cG,\delta_\cG)$ be a \emph{chain map}, i.e., a morphism of bidegree $\binom{0}{0}$ with $d(f) = 0$.  We define the \emph{cone} of $f$ to be the object
\[
\cone(f) = \left(
\cG \oplus \cF[1],
\begin{bmatrix}
\delta_\cG & f\circ s^{-1} \\
0 & -\bs(\delta_\cF)
\end{bmatrix}
\right).
\]
Then there is a natural diagram
\begin{equation}\label{eqn:dt-proto}
(\cF,\delta_\cF) \xrightarrow{f} (\cG,\delta_\cG) \to \cone(f) \to (\cF,\delta_\cF)[1].
\end{equation}
Any diagram in $\Dmix_\Gm(\cX,\bk)$, $\Dmix(\cX,\bk)$, or $\Dmix_\mon(\cX,\bk)$ that is isomorphic to (the image of) a diagram of the form~\eqref{eqn:dt-proto} is called a \emph{distinguished triangle}.

\begin{prop}\label{prop:triangulated}
The categories $\Dmix_\Gm(\cX,\bk)$, $\Dmix(\cX,\bk)$, and $\Dmix_\mon(\cX,\bk)$ are all triangulated.
\end{prop}

The proof is a minor variation on the usual proof that the homotopy category of an additive category is triangulated, as in, say,~\cite[Lemma~I.4.2 and Proposition~I.4.4]{ks}, and will be omitted. The reader who wishes to see a few details in the case of $\Dmix(\cX,\bk)$ may consult~\cite[Proposition~4.5.1]{amrw1} for a very similar situation.

\begin{lem}\label{lem:con-gen}
The categories $\Dmix_\Gm(\cX,\bk)$ and $\Dmix(\cX,\bk)$ are both generated as triangulated categories by parity sheaves.
\end{lem}
\begin{proof}
The case of $\Dmix_\Gm(\cX,\bk)$ is obvious, since it is equivalent to $\Kb\Parity_\Gm(\cX,\bk)$.

Let $(\cF,\delta)$ be an object of $\Dmix(\cX,\bk)$, and write $\cF = \bigoplus_{i \in \Z} \cF^i[-i]$.  Let $n$ be the largest integer such that $\cF^n \ne 0$.  Let $\cF' = \cF^n[-n]$, and let $\cF'' = \bigoplus_{i < n} \cF^i[-i]$. Thus $\cF \cong \cF' \oplus \cF''$.  With respect to this direct sum decomposition, the differential $\delta$ can be written as a matrix
\[
\delta =
\begin{bmatrix}
a & f \\ b & \delta''
\end{bmatrix}.
\]
But degree considerations show that $\uEnd_\sA(\cF')^1_0 = 0$ and $\uHom_\sA(\cF',\cF'')^1_0 = 0$.  Thus, $a = 0$ and $b = 0$.  Next, observe that
\[
\delta^2 + \kappa(\delta) = 
\begin{bmatrix}
0 & f \\ 0 & \delta''
\end{bmatrix}^2
+
\begin{bmatrix}
0 & \kappa(f) \\ 0 & \kappa(\delta'')
\end{bmatrix}
=
\begin{bmatrix}
0 & f\delta'' + \kappa(f) \\ 0 & (\delta'')^2 + \kappa(\delta'')
\end{bmatrix}
= 0.
\]
We deduce that $(\cF'',\delta'')$ is an object of $\Dmix(\cX,\bk)$ in its own right, and that $(\cF,\delta)$ is the cone of a morphism $f: \cF''[-1] \to \cF'$.

Note that $\cF'$ is a shift of a parity sheaf, and that the graded parity sheaf $\cF''$ has fewer nonzero components than $\cF$.  By induction on the number of nonzero components, we conclude that $\cF$ belongs to the subcategory generated by parity sheaves.
\end{proof}

\section{Monodromy and Verdier duality}
\label{sec:monodromy}

In this section, we define the functors $\Mon$ and $\Coi$ from~\eqref{eqn:intro-3cat}.  We also define Verdier duality for all three categories, and we discuss how Verdier duality interacts with the various functors.

\subsection{Construction of the monodromy functor}

Recall the inclusion map of dgg rings $\sA \to \sB$ from~\S\ref{ss:bigraded}.  For any two $\cF,\cG \in \PargrGm{\cX}$, this induces a map
\begin{equation}\label{eqn:moni-define}
i: \sA \otimes_R \uHom(\cF,\cG) \to \sB \otimes_R \uHom(\cF,\cG).
\end{equation}
If we equip these spaces with the differential $\kappa$ and $\bkappa$, respectively, then, because $\sA$ and $\sB$ are both flat over $R$, Lemma~\ref{lem:ab-qiso} implies that~\eqref{eqn:moni-define} is a quasi-isomorphism.

Now suppose that $(\cF,\delta_\cF)$ and $(\cG,\delta_\cG)$ are objects of $\DGcon{\cX}$.  Then one can consider $i(\delta_\cF) \in \sB \otimes_R \uEnd(\cF)$, and likewise for $\cG$.

\begin{lem}\label{lem:dgcon-model}
Let $(\cF,\delta_\cF)$ and $(\cG,\delta_\cG)$ be objects in $\DGcon{\cX}$.
\begin{enumerate}
\item The element $i(\delta_\cF) \in \sB \otimes_R \uEnd(\cF)$ satisfies $i(\delta_\cF)^2 + \bkappa(i(\delta_\cF)) = 0$.
\item Equip $\sA \otimes_R \uHom(\cF,\cG)$ and $\sB \otimes_R \uHom(\cF,\cG)$ with the differentials
\begin{align*}
d_\sA(f) &= \delta_\cG \circ f + (-1)^{|f|+1} f \circ \delta_\cF + \kappa(f), \\
d_\sB(f) &= i(\delta_\cG) \circ f + (-1)^{|f|+1} f \circ i(\delta_\cF) + \bkappa(f).
\end{align*}
Then $i: \sA \otimes_R \uHom(\cF,\cG) \to \sB \otimes_R \uHom(\cF,\cG)$ is a chain map.
\end{enumerate}
\end{lem}
\begin{proof}
It is easy to see that the map~\eqref{eqn:moni-define} is compatible with composition in the $\uHom$ factor, and with multiplication in the $\sA$ and $\sB$ factors.  We have already observed that it is a chain map with respect to $\kappa$ and $\bkappa$.  Both parts of the lemma follow.
\end{proof}

Next, for any $\cF \in \PargrGm{\cX}$, let
\begin{equation}\label{eqn:lvf-defn}
\Lambda^\vee \otimes \cF = \cF \oplus \cF\la -2\ra [1].
\end{equation}
This object can be equipped with a canonical isomorphism
\[
\uHom({-},\Lambda^\vee \otimes \cF) \cong \Lambda^\vee \otimes \uHom({-},\cF).
\]
For any two graded parity sheaves $\cF, \cG \in \PargrGm{\cX}$, we have the following chain of isomorphisms (the first step here relies on~\eqref{eqn:ering-end}):
\begin{multline}\label{eqn:lambda-define}
\sB \otimes_R \uHom(\cF,\cG) \cong R^\vee \otimes \uEnd(\Lambda^\vee) \otimes R \otimes_R \uHom(\cF,\cG) \\
\cong R^\vee \otimes R \otimes_R \uHom(\Lambda^\vee \otimes \cF, \Lambda^\vee \otimes \cG)
= \sS \otimes_R \uHom(\Lambda^\vee \otimes \cF, \Lambda^\vee \otimes \cG).
\end{multline}
Let $\lambda: \sB \otimes_R \uHom(\cF,\cG) \simto \sS \otimes_R \uHom(\Lambda^\vee \otimes \cF, \Lambda^\vee \otimes \cG)$ be the composition of these maps.

\begin{lem}\label{lem:dgcon-mon}
Let $\cF, \cG \in \PargrGm{\cX}$, and let
\[
\hat\delta_\cF \in (\sB \otimes_R \uEnd(\cF))^1_0,
\qquad
\hat\delta_\cG \in (\sB \otimes_R \uEnd(\cG))^1_0
\]
be elements such that $\hat\delta_\cF^2 + \bkappa(\hat\delta_\cF) = 0$ and $\hat\delta_\cG^2 + \bkappa(\hat\delta_\cG) = 0$.
\begin{enumerate}
\item The element $\lambda(\hat\delta_\cF + \omega) \in \sS \otimes_R \uEnd(\cF)$ satisfies $\lambda(\hat\delta_\cF + \omega)^2 = \Theta \id_\cF$.\label{it:dgcon-theta}
\item Equip $\sB \otimes_R \uHom(\cF,\cG)$ and $\sS \otimes_R \uHom(\cF,\cG)$ with the differentials\label{it:dgcon-chainmap}
\begin{align*}
d_\sB(f) &= \hat\delta_\cG \circ f + (-1)^{|f|+1} f \circ \hat\delta_\cF + \bkappa(f), \\
d_\sS(f) &= \lambda(\delta_\cG + \omega) \circ f + (-1)^{|f|+1} f \circ \lambda(\delta_\cF + \omega).
\end{align*}
Then $\lambda: \sB \otimes_R \uHom(\cF,\cG) \simto \sS \otimes_R \uHom(\Lambda^\vee \otimes \cF, \Lambda^\vee \otimes \cG)$ is an isomorphism of chain complexes.
\end{enumerate}
\end{lem}
\begin{proof}
\eqref{it:dgcon-theta}~We have
\[
(\hat\delta_\cF + \omega)^2 = \hat\delta_\cF^2 + \omega\hat\delta_\cF + \hat\delta_\cF\omega + \omega^2
= \hat\delta_\cF^2 + \bkappa(\hat\delta_\cF) + \Theta = \Theta.
\]
\eqref{it:dgcon-chainmap}~This follows from the formula $\bkappa(f) = \omega \circ f + (-1)^{|f|+1}f \circ \omega$ and the observation that $\lambda$ is compatible with composition.
\end{proof}

We are now ready to define a functor
\[
\Mon: \DGcon{\cX} \to \DGmon{\cX}.
\]
On objects, it is given by $\Mon(\cF,\delta) = (\Lambda^\vee \otimes \cF, \lambda(i(\delta) + \omega))$.   On morphism spaces, it is given by
\begin{equation}\label{eqn:mon-mor-define}
\Mon = \lambda \circ i: \sA \otimes_R \uHom(\cF,\cG) \to \sS \otimes_R \uHom(\cF,\cG).
\end{equation}
Lemmas~\ref{lem:dgcon-model} and~\ref{lem:dgcon-mon} tell us that $\Mon(\cF,\delta)$ is a well-defined object of $\DGmon{\cX}$, and that the map~\eqref{eqn:mon-mor-define} is indeed a chain map.  After passing to homotopy categories, we obtain a functor
\[
\Mon: \Dmix(\cX,\bk) \to \Dmix_\mon(\cX,\bk).
\]

Let us try to make the description of $\Mon$ more concrete. As part of the map $\lambda \circ i$ from~\eqref{eqn:moni-define} and~\eqref{eqn:lambda-define}, we have an inclusion map
\begin{equation}\label{eqn:mon-incl}
\jmath: \uHom(\cF,\cG) \to \uHom(\Lambda^\vee \otimes \cF, \Lambda^\vee \otimes \cG)
\qquad\text{given by}\qquad
\jmath(f) = \id_{\Lambda^\vee} \otimes f.
\end{equation}
However, with respect to the decomposition~\eqref{eqn:lvf-defn} (and its analogue for $\cG$), there are signs involved: we claim that
\[
\jmath(f) = \begin{bmatrix} f & \\ & (-1)^{|f|}f \end{bmatrix}.
\]
To see this, consider the unit and counit maps $\iota_{\Lambda^\vee}: \bk \to \Lambda^\vee$ and $\varepsilon_{\Lambda^\vee}: \Lambda^\vee \to \bk$.  The upper-left entry of $\jmath(f)$ is the element of $\uEnd(\cF)$ given by
\[
(\varepsilon_{\Lambda^\vee} \otimes \id_\cG)(\id_{\Lambda^\vee} \otimes f)(\iota_{\Lambda^\vee} \otimes \id_\cF) = \id_{\bk} \otimes f,
\]
while the lower-right entry is given by
\[
(\varepsilon_{\Lambda^\vee}\bxi \otimes \id_\cG)(\id_{\Lambda^\vee} \otimes f)(\br\iota_{\Lambda^\vee} \otimes \id_\cF) = (-1)^{|\br||f|} (\id_{\bk} \otimes f) = (-1)^{|f|} f.
\]
With respect to the decomposition~\eqref{eqn:lvf-defn}, the maps $\bxi$ and $\br$ can be written as $\bxi = \left[\begin{smallmatrix} 0 & \id \\ 0 & 0 \end{smallmatrix}\right]$ and  $\br = \left[\begin{smallmatrix} 0 & 0 \\ \id & 0 \end{smallmatrix}\right]$, respectively, so $\omega = \left[\begin{smallmatrix}  & \sr \\ \xi &  \end{smallmatrix}\right]$.

Now let $(\cF,\delta)$ be an object of $\DGcon{\cX}$.  Write $\delta$ as $\delta_0 + \bxi \delta_1$, where $\delta_0 \in \uEnd(\cF)^1_0$, and $\delta_1 \in \uEnd(\cF)^0_{-2}$.  Then
\[
\lambda(i(\delta) + \omega) =
\begin{bmatrix}
\delta_0 & \\ & -\delta_0
\end{bmatrix}
+
\begin{bmatrix}
0 & \id \\ 0 & 0
\end{bmatrix}
\begin{bmatrix}
\delta_1 & \\ & \delta_1
\end{bmatrix}
+
\begin{bmatrix}
& \sr \\ \xi &  
\end{bmatrix}.
\]
Thus $\Mon(\cF,\delta)$ can be written as
\begin{multline}\label{eqn:mon-concrete}
\Mon(\cF, \delta_0 + \bxi \delta_1) = \\
\left( \cF \oplus \cF\la -2\ra[1],
\begin{bmatrix}
\delta_0  & \delta_1 + \sr \\
\xi & -\delta_0
\end{bmatrix}
\right) 
\qquad\text{or}\qquad
\begin{tikzcd}[ampersand replacement=\&]
\cF \ar[loop above, out=120, in=60, distance=20, "{\sst[1]}"{description, pos=0.18}, "\delta_0"]
  \ar[r, shift left=1.5, "{\sst[1]}" description, "\xi" above=2] \&
\cF\la -2\ra[1] \ar[loop above, out=120, in=60, distance=20, "{\sst[1]}"{description, pos=0.18}, "-\delta_0"]
  \ar[l, shift left=1.5, "{\sst[1]}" description, "\delta_1 + \sr" below=2].
\end{tikzcd}
\end{multline}
If we write $f \in (\sA \otimes_R \uHom(\cF,\cG))^i_j$ as $f_0 + \bxi f_1$ with $f_0 \in \uHom(\cF,\cG)^i_j$ and $f_1 \in \uHom(\cF,\cG)^{i-1}_{j-2}$, then 
\[
\Mon(f) = \begin{bmatrix}
f_0 & (-1)^{|f|+1} f_1 \\ & (-1)^{|f|}f_0
\end{bmatrix}.
\]

\begin{rmk}\label{rmk:mon-concrex}
For an explicit example, applying $\Mon$ to the object from Example~\ref{ex:con-der} yields the object from Example~\ref{ex:mon-der2}.
\end{rmk}

\begin{lem}\label{lem:mon-tri}
The functor $\Mon: \Dmix(\cX,\bk) \to \Dmix_\mon(\cX,\bk)$ is triangulated.
\end{lem}
\begin{proof}
Using the explicit formula for $\Mon$ given above along with the descriptions of the shift functor $[1]$ from Section~\ref{sec:triangulated}, we have
\begin{align*}
\Mon((\cF,\delta_\cF)[1]) &= \left(\cF[1] \oplus \cF\la -2\ra[2], 
\left[\begin{smallmatrix}
-\delta_{\cF,0}  & \delta_{\cF,1} + \sr \\
\xi & \delta_{\cF,0}
\end{smallmatrix}\right]
\right),
\\
\Mon(\cF,\delta_\cF)[1] &= \left(\cF[1] \oplus \cF\la -2\ra[2], 
\left[\begin{smallmatrix}
-\delta_{\cF,0}  & -\delta_{\cF,1} - \sr \\
-\xi & \delta_{\cF,0}
\end{smallmatrix}\right]
\right),
\end{align*}
where $\delta_\cF = \delta_{\cF,0} + \bxi \delta_{\cF,1}$.  Then the map
\begin{equation}\label{eqn:monshift-map}
\left[\begin{smallmatrix}
\id & \\ & -\id
\end{smallmatrix}\right]: \cF[1] \oplus \cF\la-2\ra[2] \to \cF[1] \oplus \cF\la -2\ra[2]
\end{equation}
is clearly a chain isomorphism $\Mon((\cF,\delta_\cF)[1]) \to \Mon(\cF,\delta_\cF)[1]$.

Next, let $(\cG,\delta_\cG)$ be another object of $\Dmix(\cX,\bk)$, and write $\delta_\cG = \delta_{\cG,0} + \bxi \delta_{\cG,1}$. Let $f: (\cF,\delta_\cF) \to (\cG,\delta_\cG)$ be a chain map, and write $f = f_0 + \bxi f_1$.  Recall that the cone of $f$ is given by
\[
\cone(f) = \left( \cG \oplus \cF[1], 
\left[\begin{smallmatrix}
\delta_{\cG,0} & f_0 \\ & -\delta_{\cF_0}
\end{smallmatrix}\right]
+ \bxi
\left[\begin{smallmatrix}
\delta_{\cG,1} & f_1 \\ & \delta_{\cF_1}
\end{smallmatrix}\right]
\right)
\]
We have
\[
\Mon(\cone(f)) = \left(
\cG \oplus \cF[1] \oplus \cG\la -2\ra[1] \oplus \cF \la -2\ra[2],
\left[\begin{smallmatrix}
\delta_{\cG,0} & f_0 & \delta_{\cG,1} + \sr & f_1 \\
 & -\delta_{\cF,0} & & \delta_{\cF,1} + \sr \\
\xi & & -\delta_{\cG,0} & -f_0 \\
& \xi & & \delta_{\cF,0}
\end{smallmatrix}\right]
\right)
\]
and
\[
\cone(\Mon(f)) =
\left(
\cG \oplus \cG\la -2\ra[1] \oplus \cF[1] \oplus \cF\la -2\ra[2],
\left[\begin{smallmatrix}
\delta_{\cG,0} & \delta_{\cG,1} + \sr & f_0 & -f_1 \\
\xi & -\delta_{\cG,0} & & f_0 \\
& & -\delta_{\cF,0} & -\delta_{\cF,1} - \sr \\
& & -\xi & \delta_{\cF,0}
\end{smallmatrix}\right]
\right).
\]
Then it is easily checked that the map
\begin{equation}\label{eqn:moncone-map}
\left[\begin{smallmatrix}
\id & & & \\ & & \id & \\ & \id & & \\ & & & -\id
\end{smallmatrix}
\right]
:
\begin{aligned}
\cG \oplus \cF[1] \oplus \cG\la -2\ra[1] &\oplus \cF \la -2\ra[2] \to \\
\cG &\oplus \cG\la -2\ra[1] \oplus \cF[1] \oplus \cF\la -2\ra[2]
\end{aligned}
\end{equation}
provides a natural isomorphism $\Mon(\cone(f)) \simto \cone(\Mon(f))$, and moreover that the two natural transformations we have defined give rise to a commutative diagram
\[
\begin{tikzcd}
\Mon(\cF,\delta_\cF) \ar[r, "\Mon(f)"] \ar[d,equal] &
  \Mon(\cG,\delta_\cG) \ar[r] \ar[d,equal] &
  \Mon(\cone(f)) \ar[d, "\text{\eqref{eqn:moncone-map}}"] \ar[r] &
  \Mon(\cF[1]) \ar[d, "\text{\eqref{eqn:monshift-map}}"] \\
\Mon(\cF,\delta_\cF) \ar[r, "\Mon(f)"] &
  \Mon(\cG,\delta_\cG) \ar[r] &
  \cone(\Mon(f)) \ar[r] &
  \Mon(\cF)[1]
\end{tikzcd}
\]
Thus, $\Mon$ is triangulated.
\end{proof}

\begin{prop}\label{prop:mon-ff}
The functor $\Mon: \Dmix(\cX,\bk) \to \Dmix_\mon(\cX,\bk)$ is fully faithful.
\end{prop}
\begin{proof}
Since $\Mon$ is triangulated, it is enough to check that
\begin{equation}\label{eqn:mon-ff}
\Hom((\cF,\delta_\cF),(\cG,\delta_\cG)) \to \Hom(\Mon(\cF,\delta_\cF), \Mon(\cG,\delta_\cG))
\end{equation}
is an isomorphism when $(\cF,\delta_\cF)$ and $(\cG,\delta_\cG)$ belong to some class of objects that generates $\Dmix(\cX,\bk)$.  For instance, by Lemma~\ref{lem:con-gen}, we may assume that they are both shifts of parity sheaves with zero differential.  Since $\delta_\cF = 0$ and $\delta_\cG = 0$, the differentials $d_\sA$ and $d_\sB$ from Lemma~\ref{lem:dgcon-model} reduce to $\kappa$ and $\bkappa$, respectively. In this special case, we have seen that the map~\eqref{eqn:moni-define} is a quasi-isomorphism.  Since $\lambda$ is an isomorphism, the map~\eqref{eqn:mon-mor-define} is a quasi-isomorphism as well, so~\eqref{eqn:mon-ff} is an isomorphism.
\end{proof}

\begin{rmk}\label{rmk:mon-nilp}
An immediate and striking consequence is that Definition~\ref{defn:monodromy} can be transferred to the constructible category: for any $\cF \in \Dmix(\cX,\bk)$, we have a canonical map
\[
\Nilp_\cF: \cF \to \cF\la 2\ra,
\]
that commutes with all morphisms in $\Dmix(\cX,\bk)$.  By Remark~\ref{rmk:hom-nilp}, this map is nilpotent: $\Nilp_\cF^k = 0$ for $k \gg 0$.  (In contrast, Example~\ref{ex:mon-der1} shows that in $\Dmix_\mon(\cX,\bk)$, $\Nilp_\cF$ need not be nilpotent.)
\end{rmk}

The following statement describes the monodromy endomorphism in terms intrinsic to $\Dmix(\cX,\bk)$.  It resembles (up to a sign) the ``monodromy action'' discussed in~\cite[\S4.7]{amrw1}.

\begin{prop}\label{prop:constr-monodromy}
Let $(\cF,\delta) \in \Dmix(\cX,\bk)$, and write its differential as $\delta = \delta_0 + \bxi \delta_1$, where $\delta_0 \in \uEnd(\cF)^1_0$, and $\delta_1 \in \uEnd(\cF)^0_{-2}$.  Then the monodromy endomorphism is given by $\Nilp_\cF = -\delta_1$.
\end{prop}
This proposition implicitly asserts that $\delta_1$ is a chain map, i.e., that $\delta\delta_1 - \delta_1\delta + \kappa(\delta_1) = 0$.  This follows easily from the fact that $\delta^2 + \kappa(\delta) = 0$.
\begin{proof}
It is enough to show that $\Mon(-\delta_1) = \Nilp_{\Mon(\cF)}$ in $\Dmix_\mon(\cX,\bk)$.  Let $h \in (\sS \otimes_R \uEnd(\cF \oplus \cF\la -2\ra[1]))^{-1}_{-2}$ be the map given by $h =
[\begin{smallmatrix}
0 & 0 \\
\id & 0
\end{smallmatrix}]$.
Then
\begin{multline*}
d(h)
=
\begin{bmatrix}
\delta_0  & \delta_1 + \sr \\
\xi & -\delta_0
\end{bmatrix}
\begin{bmatrix}
0 & 0 \\
\id & 0
\end{bmatrix}
+
\begin{bmatrix}
0 & 0 \\
\id & 0
\end{bmatrix}
\begin{bmatrix}
\delta_0  & \delta_1 + \sr \\
\xi & -\delta_0
\end{bmatrix} 
=
\begin{bmatrix}
\delta_1+\sr & \\
& \delta_1 + \sr
\end{bmatrix} \\
= \Mon(\delta_1) + \sr\cdot\id 
= \Mon(\delta_1) + \Nilp_{\Mon(\cF)}.
\end{multline*}
Thus, the chain maps $\Mon(-\delta_1)$ and $\Nilp_{\Mon(\cF)}$ are homotopic, as desired.
\end{proof}

\subsection{Verdier duality}

Let us denote by $\D_\pari: \Parity_\Gm(\cX,\bk) \to \Parity_\Gm(\cX,\bk)$ the ordinary Verdier duality functor on parity sheaves.  We will use this to define a version of Verdier duality for each of the three derived categories from~\S\ref{sec:3derived}.  As a first step, we extend $\D_\pari$ to graded parity sheaves by the formula
\[
\D_\pari\left(\bigoplus_{i \in \Z} \cF^i[-i]\right) = \bigoplus_{i\in \Z} (\D\cF^i)[i].
\]
For $\cF,\cG \in \PargrGm{\cX}$, $\D_\pari$ induces an isomorphism of bigraded $\bk$-modules
\begin{equation}\label{eqn:verdier-uhom}
\D_\pari: \uHom(\cF,\cG) \to \uHom(\D_\pari\cG,\D_\pari\cF).
\end{equation}
The ring automorphism of $\coh^\bullet_\Gm(\pt,\bk)$ induced by $\D_\pari$ is the identity map.  More generally, the map~\eqref{eqn:verdier-uhom} is a homomorphism of $R$-modules.

The easiest case is the equivariant derived category: we define
\[
\D: \Dmix_\Gm(\cX,\bk)^\op \to \Dmix_\Gm(\cX,\bk) 
\qquad\text{by}\qquad
\D(\cF,\delta) = (\D_\pari\cF, \D_\pari\delta).
\]
Next, consider the map
\[
\id_\Lambda \otimes \D_\pari: \sA \otimes_R \uHom(\cF,\cG) \to \sA \otimes_R \uHom(\D_\pari\cG,\D_\pari\cF).
\]
This is a chain map with respect to the differential $\kappa$.  This observation lets us define
\[
\D: \Dmix(\cX,\bk) \to \Dmix(\cX,\bk)
\qquad\text{by}\qquad
\D(\cF,\delta) = (\D_\pari\cF, (\id_\Lambda \otimes \D_\pari)(\delta)).
\]
By construction, we clearly have
\[
\D \circ \For \cong \For \circ \D.
\]

The monodromic category is a bit trickier, because we would like $\Mon$ to commute with Verdier duality as well.  We define
\[
\D: \Dmix(\cX,\bk) \to \Dmix(\cX,\bk)
\quad\text{by}\quad
\D(\cF,\delta) = ((\D_\pari\cF)\la -2\ra[1], \D_\pari(\delta)\la -2\ra[2]).
\]

\begin{lem}
For $\cF \in \Dmix_\mon(\cX,\bk)$, there is a natural isomorphism $\Mon(\D\cF) \cong \D\Mon(\cF)$.
\end{lem}
\begin{proof}
Using the formula from~\eqref{eqn:mon-concrete}, we see that
\begin{align*}
\Mon(\D(\cF, \delta_0 + \bxi\delta_1))
&=
\left((\D\cF) \oplus (\D\cF)\la -2\ra[1],
\left[\begin{smallmatrix}
\D_\pari(\delta_0) & \D_\pari(\delta_1) + \sr \\
\xi & -\D_\pari(\delta_0)
\end{smallmatrix}\right]
\right),
\\
\D\Mon(\cF, \delta_0 + \bxi\delta_1)
&=
\left((\D\cF) \oplus (\D\cF)\la -2\ra[1],
\left[\begin{smallmatrix}
-\D_\pari(\delta_0) & \D_\pari(\delta_1) + \sr \\
\xi & \D_\pari(\delta_0)
\end{smallmatrix}\right]
\right)
\end{align*}
Write $\D\cF$ as $\bigoplus_{i \in \Z} \cG^i[-i]$ with $\cG^i \in \Parity_\Gm(\cX,\bk)$, and let $q \in \uEnd(\D\cF)$ be the map $q = \sum(-1)^i \id_{\cG^i}[-i]$.  It is easy to see from degree considerations that
\[
q \circ \D_\pari(\delta_0) = -\D_\pari(\delta_0) \circ q,
\qquad
q \circ \D_\pari(\delta_1) = \D_\pari(\delta_1) \circ q.
\]
It follows that the map
\[
\left[\begin{smallmatrix}
q & \\ & q
\end{smallmatrix}\right]:
(\D\cF) \oplus (\D\cF)\la -2\ra[1] \to (\D\cF) \oplus (\D\cF)\la -2\ra[1]
\]
defines a natural isomorphism $\Mon \circ \D \cong \D \circ \Mon$.
\end{proof}

It is left to the reader to check that all three versions of Verdier duality are triangulated functors.  They also satisfy
\[
\D \circ \D \cong \id.
\]
In particular, all three versions of $\D$ are equivalences of categories.

\subsection{Invariants and coinvariants of monodromy}

For $\cF, \cG \in \PargrGm{\cX}$, the counit map $\varepsilon_{R^\vee}: R^\vee \to \bk$ induces a map of $\uHom$-spaces that we denote by
\[
\Coi: \sS \otimes_R \uHom(\cF,\cG) \to \uHom(\cF,\cG).
\]
If $(\cF, \delta)$ is an object of $\DGmon{\cX}$, then it is easy to see that $(\cF, \Coi(\delta))$ is an object of $\DGeq{\cX}$.  We therefore obtain a functor denoted by $\Coi: \DGmon{\cX} \to \DGeq{\cX}$.  After passing to homotopy categories, we obtain a functor
\[
\Coi: \Dmix_\mon(\cX,\bk) \to \Dmix_\Gm(\cX,\bk),
\]
called the \emph{coinvariants of monodromy functor}, or simply the \emph{coinvariants functor}.

Unlike $\For$ and $\Mon$, the coinvariants functor does not commute with Verdier duality.  We define
\[
\Inv: \Dmix_\mon(\cX,\bk) \to \Dmix_\Gm(\cX,\bk)
\qquad\text{by}\qquad
\Inv = \D \circ \Coi \circ \D.
\]
It is called the \emph{invariants of monodromy functor}, or simply the \emph{invariants functor}.  The following fact is immediate from the definitions.

\begin{lem}
For $\cF \in \Dmix_\mon(\cX,\bk)$, there is a natural isomorphism $\Inv(\cF) \cong \Coi(\cF) \la 2\ra[-1]$.
\end{lem}

The following proposition gives a key property of these functors.

\begin{prop}\label{prop:coi-adjoint}
The functor $\Mon \circ \For: \Dmix_\Gm(\cX,\bk) \to \Dmix_\mon(\cX,\bk)$ is right adjoint to $\Coi$ and left adjoint to $\Inv$.
\end{prop}
\begin{proof}
For brevity, let $\For' = \Mon \circ\For: \Dmix_\Gm(\cX,\bk) \to \Dmix_\mon(\cX,\bk)$.  It is enough to prove that $\For'$ is right adjoint to $\Coi$, as the other part of the proposition would then follow by Verdier duality.

This functor can be described by omitting the term ``$\delta_1$'' from the formula in~\eqref{eqn:mon-concrete}: for $(\cF,\delta_\cF) \in \Dmix_\Gm(\cX,\bk)$, we have
\[
\For'(\cF,\delta_\cF) =  
\begin{tikzcd}
\cF \ar[loop above, out=120, in=60, distance=20, "{\sst[1]}"{description, pos=0.18}, "\delta_\cF"]
  \ar[r, shift left=1.5, "{\sst[1]}" description, "\xi" above=2] &
\cF\la -2\ra[1] \ar[loop above, out=120, in=60, distance=20, "{\sst[1]}"{description, pos=0.18}, "-\delta_\cF"]
  \ar[l, shift left=1.5, "{\sst[1]}" description, "\sr" below=2].
\end{tikzcd}
\]
Then $\Coi(\For'(\cF,\delta_\cF)) \in \Dmix_\Gm(\cX,\bk)$ is obtained by suppressing the arrow labelled $\sr$.  Define a map $\epsilon: \Coi(\For'(\cF,\delta_\cF)) \to (\cF,\delta_\cF)$ by the diagram
\[
\begin{tikzcd}
\cF \ar[loop left, distance=30, "{\sst[1]}"{description, pos=0.18}, "\delta_\cF"]
  \ar[d, "{\sst[1]}" description, "\xi" left=2] \ar[r, "\id"] &
\cF \ar[loop right, distance=30, "{\sst[1]}"{description, pos=0.18}, "\delta_\cF"] \\
\cF\la -2\ra[1] \ar[loop left, distance=30, "{\sst[1]}"{description, pos=0.18}, "-\delta_\cF"]
\end{tikzcd}
.
\]

Next, let $(\cG, \delta_\cG) \in \Dmix_\mon(\cX,\bk)$.  Expand $\delta_\cG \in R^\vee \otimes \uEnd(\cG)$ as
\[
\delta_\cG = \delta_{\cG,0} + \sr \delta_{\cG,1},
\]
where $\delta_{\cG,0} \in \uEnd(\cG)^1_0$, and $\delta_{\cG,1} \in (R^\vee \otimes \uEnd(\cG))^1_2$.  The equation $\delta_\cG^2 = \Theta = \sr\xi$ implies that
\begin{equation}\label{eqn:coi-adj-check}
\delta_{\cG,0}^2 = 0
\qquad\text{and}\qquad
\delta_{\cG,0}\delta_{\cG,1} + \delta_{\cG,1}\delta_{\cG,0} + \sr\delta_{\cG,1}^2 = \xi\id.
\end{equation}
We have $\Coi(\cG,\delta_\cG) = (\cG,\delta_{\cG,0})$, and
\[
\For'(\Coi(\cG,\delta_\cG)) =
\begin{tikzcd}
\cG \ar[loop above, out=120, in=60, distance=20, "{\sst[1]}"{description, pos=0.18}, "\delta_{\cG,0}"]
  \ar[r, shift left=1.5, "{\sst[1]}" description, "\xi" above=2] &
\cG\la -2\ra[1] \ar[loop above, out=120, in=60, distance=20, "{\sst[1]}"{description, pos=0.18}, "-\delta_{\cG,0}"]
  \ar[l, shift left=1.5, "{\sst[1]}" description, "\sr" below=2].
\end{tikzcd}
\]
Define a map $\eta: (\cG,\delta_\cG) \to \For'(\Coi(\cG,\delta_\cG))$ by
\[
\begin{tikzcd}
\cG \ar[loop left, distance=30, "{\sst[1]}"{description, pos=0.18}, "\delta_\cG = \delta_{\cG,0} + \sr \delta_{\cG_1}"]
  \ar[r, "\id"] \ar[dr, "\delta_{\cG,1}"'] &
\cG \ar[loop right, distance=30, "{\sst[1]}"{description, pos=0.18}, "\delta_{\cG_0}"]
  \ar[d, shift left=1.5, "{\sst[1]}" description, "\xi" right=2] \\
&
\cG\la -2\ra[1] \ar[loop right, distance=30, "{\sst[1]}"{description, pos=0.18}, "-\delta_{\cG,0}"]
  \ar[u, shift left=1.5, "{\sst[1]}" description, "\sr" left=2]
\end{tikzcd}
.
\]
It follows from~\eqref{eqn:coi-adj-check} that this is indeed a chain map.  
Straightforward calculations then show that $\epsilon$ and $\eta$ satisfy the counit--unit equations.
\end{proof}

\section{Recollement}
\label{sec:recollement}

Let $i: \cZ \hookrightarrow \cX$ be the inclusion of a closed union of strata, and let $j: \cU \hookrightarrow \cX$ be the complementary open inclusion.  In the ordinary (nonnmixed) derived category, the functors $i_*$ and $j^*$ take parity sheaves to parity sheaves, so we get induced functors
\begin{equation}\label{eqn:recolle-easy}
\begin{aligned}
i_* &:\PargrGm{\cZ} \to \PargrGm{\cX}, \\
j^*&: \PargrGm{\cX} \to \PargrGm{\cU}.
\end{aligned}
\end{equation}
It is straightforward to see that these extend to triangulated functors of the three kinds of derived categories from Section~\ref{sec:3derived}.  The goal of this section is to prove the following theorem.

\begin{thm}\label{thm:recollement}
Let $\sD$ stand for one of $\Dmix_\Gm$, $\Dmix$, or $\Dmix_\mon$.  Let $i: \cZ \hookrightarrow \cX$ be the inclusion of a closed union of strata, and let $j: \cU \hookrightarrow \cX$ be the complementary open inclusion.  Then the functors $i_*$ and $j^*$ both admit left and right adjoints, and these adjoints form a recollement diagram:
\[
\begin{tikzcd}[column sep=large]
\sD(\cZ,\bk) \ar[r, "i_*" description] &
\sD(\cX,\bk) \ar[r, "j^*" description] 
  \ar[l, bend left=30, "i^!" description] \ar[l, bend right=30, "i^*" description] &
\sD(\cU,\bk)
  \ar[l, bend left=30, "j_*" description] \ar[l, bend right=30, "j_!" description] 
\end{tikzcd}
\]
\end{thm}

We first require some preliminaries involving the case where $\cZ$ is a single stratum.

\begin{lem}\label{lem:closed-ses}
Let $i: \cX_s \hookrightarrow \cX$ be the inclusion of a closed stratum.  Let $\cU = \cX \smallsetminus \cX_s$, and let $j: \cU \hookrightarrow \cX$ be the inclusion map.  Let $\sC$ be a flat $R$-algebra.  For any two graded parity sheaves $\cF, \cG \in \PargrGm{\cX}$, there is a natural short exact sequence of bigraded $\bk$-modules
\[
0 \to \sC \otimes_R \uHom(\cF,i_*i^!\cG) \to \sC \otimes_R \uHom(\cF,\cG) \to \sC \otimes_R \uHom(j^*\cF,j^*\cG) \to 0.
\]
\end{lem}
In applications, the ring $\sC$ will be one of $R$, $\sA$, or $\sS$.  
\begin{proof}
In the special case where $\sC = R$, this is essentially a restatement of~\cite[Proposition~2.6]{jmw}.  The general case follows because $\sC$ is flat over $R$.
\end{proof}

\begin{lem}\label{lem:open-gen}
Let $\cU = \cX \smallsetminus \cX_s$ be the complement of a closed stratum, and let $j: \cU \to \cX$ be the inclusion map.  Then the functor $j^*: \sD(\cX,\bk) \to \sD(\cU,\bk)$ is essentially surjective.
\end{lem}
\begin{proof}
We will prove this in detail for $\Dmix_\mon$.  To obtain the proof for $\Dmix_\Gm$ or $\Dmix$, replace all mentions of $\sS$ by $R$ or $\sA$, as appropriate, and replace all mentions of $\Theta$ by $0$.

Consider an object $(\cF,\delta)$ in $\Dmix_\mon(\cU,\bk)$.  The assumptions in Section~\ref{ss:parity} imply that every parity sheaf on $\cU$ extends to a parity sheaf on $\cX$. The same holds for graded parity sheaves: we can find an object $\tilde\cF \in \PargrGm{\cX}$ together with an isomorphism $j^*\tilde\cF \simto \cF$.  By Lemma~\ref{lem:closed-ses}, the map
\begin{equation}\label{eqn:open-gen-surj}
\sS \otimes_R \uEnd(\tilde\cF) \to \sS \otimes_R \uEnd(\cF)
\end{equation}
is surjective.  Choose an element $\tilde\delta \in \sS \otimes_R \uEnd(\tilde\cF)^1_0$ such that $j^*\tilde\delta$ is identified with $\delta$.  Since $\delta^2 = \Theta$, we see that $\tilde\delta^2 - \Theta$ lies in the kernel of~\eqref{eqn:open-gen-surj}.  By Lemma~\ref{lem:closed-ses} again, there exists a unique element $\delta' \in \sS \otimes_R \uHom(\tilde\cF, i_*i^!\tilde\cF)^2_0$ such that
\begin{equation}\label{eqn:open-gen-dp}
\epsilon \circ \delta' = \tilde\delta^2 - \Theta,
\end{equation}
where $\epsilon: i_*i^!\tilde\cF \to \tilde\cF$ is the adjunction map.

Let $\cG = \tilde\cF \oplus i_*i^!\tilde\cF[1]$, and define $\delta_\cG \in \uEnd(\cG)^1_0$ by
\[
\delta_\cG =
\begin{bmatrix}
\tilde\delta & \epsilon \\ -\delta' & -i_*i^!\tilde\delta
\end{bmatrix}
\]
We will show below that $(\cG,\delta_\cG)$ is an object of $\Dmix_\mon(\cX,\bk)$.  This claim implies the lemma, since we will then clearly have $j^*(\cG,\delta_\cG) \cong (\cF,\delta)$.

Observe first that
\begin{equation}\label{eqn:open-gen-diff}
\delta_\cG^2 =  
\begin{bmatrix}
\tilde\delta & \epsilon \\ -\delta' & -i_*i^!\tilde\delta
\end{bmatrix}^2
=
\begin{bmatrix}
\tilde\delta^2 - \epsilon\delta' & \tilde \delta\epsilon - \epsilon (i_*i^!\tilde\delta) \\
-\delta'\tilde\delta + (i_*i^!\tilde\delta)\delta' & -\delta'\epsilon + (i_*i^!\tilde\delta)^2 
\end{bmatrix}.
\end{equation}
The upper left entry of this matrix is $\Theta$, by~\eqref{eqn:open-gen-dp}.  The upper right entry is $0$ because $\epsilon$ is a natural transformation, i.e., because the following diagram commutes:
\begin{equation}\label{eqn:open-gen-comm}
\begin{tikzcd}
i_*i^!\tilde\cF \ar[r, "\epsilon"] \ar[d, "i_*i^!\tilde \delta"'] & \tilde\cF \ar[d, "\tilde\delta"] \\
i_*i^!\tilde\cF[1] \ar[r, "\epsilon"] & \tilde\cF[1]
\end{tikzcd}
\end{equation}
Next, by~\eqref{eqn:open-gen-dp}, we have $\epsilon \delta'\tilde \delta = \tilde \delta^3 - \Theta\tilde\delta$.  On the other hand, using~\eqref{eqn:open-gen-comm}, we have $\epsilon(i_*i^!\tilde\delta)\delta' = \tilde\delta \epsilon \delta' = \tilde\delta^3 - \Theta\tilde\delta$.  We have shown that
\[
\epsilon \delta'\tilde \delta = \epsilon(i_*i^!\tilde\delta)\delta'.
\]
By Lemma~\ref{lem:closed-ses}, composition with $\epsilon$ gives an injective map $\sS \otimes_R \uHom(\cF,i_*i^!\cF) \to \sS \otimes_R \uEnd(\cF)$.  The equation above therefore implies that $\delta'\tilde \delta = (i_*i^!\tilde\delta)\delta'$.  In other words, the bottom left entry of~\eqref{eqn:open-gen-diff} vanishes.

Finally, applying $i_*i^!$ to~\eqref{eqn:open-gen-dp}, we find that $i_*i^!\delta' = (i_*i^!\tilde\delta)^2 - \Theta$.  The commutative square
\[
\begin{tikzcd}[column sep=large]
i_*i^!\tilde\cF \ar[r, "\epsilon_\cF"] \ar[d, "i_*i^!\delta'"'] & \tilde\cF \ar[d, "\delta'"] \\
i_*i^!\tilde\cF[2] \ar[r, equal, "\id = \epsilon_{i_*i^!\cF}"] & i_*i^!\tilde\cF[2]
\end{tikzcd}
\]
shows that $i_*i^!\delta' = \delta'\epsilon$.  Together, these observations show that the bottom right entry of~\eqref{eqn:open-gen-diff} is $\Theta$.  We have shown that $\delta_\cG^2 = \Theta$, as desired.
\end{proof}

\begin{lem}\label{lem:recollement1}
In the special case where $\cZ$ consists of a single stratum, Theorem~\ref{thm:recollement} holds.
\end{lem}
\begin{proof}
In this proof, we let
\[
\sC =
\begin{cases}
R & \text{if we are working in $\Dmix_\Gm(\cX,\bk)$,} \\
\sA & \text{if we are working in $\Dmix(\cX,\bk)$,} \\
\sS & \text{if we are working in $\Dmix_\mon(\cX,\bk)$.}
\end{cases}
\]
Because $\cZ$ consists of a single stratum, the ordinary (nonmixed) functors $i^*$ and $i^!$ take parity sheaves to parity sheaves, so as in~\eqref{eqn:recolle-easy}, there are induced functors $i^*,i^!: \sD(\cX,\bk) \to \sD(\cZ,\bk)$.  Note that if $\cF \in \sD(\cZ,\bk)$ and $\cG \in \sD(\cX,\bk)$, then there is a natural isomorphism of dgg $\bk$-modules
\[
\sC \otimes_R \uHom(\cF, i_*\cG) \cong \sC \otimes_R \uHom(i^*\cF,\cG),
\]
and hence of $\bk$-modules $\Hom(\cF,i_*\cG) \cong \Hom(i^*\cF,\cG)$.  In other words, the functor $i^*: \sD(\cX,\bk) \to \sD(\cZ,\bk)$ is indeed left adjoint to $i_*: \sD(\cZ,\bk) \to \sD(\cX,\bk)$.  Similarly, $i^!$ is right adjoint to $i_*$.

We will construct the right adjoint $j_!$ to $j^*$; we will show that the adjunction map $\id \to j^*j_!$ is an isomorphism; and we will show that for every object $\cF \in \sD(X,\bk)$, there is a distinguished triangle
\[
j_!j^*\cF \to \cF \to i_*i^*\cF \to.
\]
The proofs for the corresponding assertions about the right adjoint $j_*$ to $j^*$ are similar and will be omitted.

For any $\cF \in \sD(\cX,\bk)$, let $\cF^+$ denote the cocone of the adjunction map $\cF \to i_*i^*\cF$. In the following paragraphs, we will prove a number of assertions about $\cF^+$.

{\it Step 1. For any $\cG \in \sD(\cZ,\bk)$, we have $\Hom(\cF^+,i_*\cG = 0)$.}  The natural map $\cF \to i_*i^*\cF$ induces a map
\[
\sC \otimes_R \uHom(i_*i^*\cF,i_*\cG) \to \sC \otimes_R \uHom(\cF,i_*\cG).
\]
By the usual adjunction properties of $i_*$ and $i^*$, this is an isomorphism of bigraded $\bk$-modules (at the level of graded parity sheaves), and hence also of chain complexes.  Next, the distinguished triangle $\cF^+ \to \cF \to i_*i^*\cF \to$ gives rise to a long exact sequence
\begin{multline*}
\Hom(i_*i^*\cF, i_*\cG) \to \Hom(\cF,i_*\cG) \to \Hom(\cF^+,i_*\cG) \to \\
\Hom(i_*i^*\cF, i_*\cG[1]) \to \Hom(\cF,i_*\cG[1]).
\end{multline*}
The discussion above implies that the first and last maps here as isomorphisms.  It follows that $\Hom(\cF^+,i_*\cG) = 0$.

{\it Step 2. For any $\cG \in \sD(\cX,\bk)$, the natural map $\Hom(\cF^+,\cG) \to \Hom(j^*\cF^+,j^*\cG)$ is an isomorphism.}  By Lemma~\ref{lem:closed-ses}, we have a short exact sequence of chain complexes
\[
0 \to \sC \otimes_R \uHom(\cF^+, i_*i^!\cG) \to \sC \otimes_R \uHom(\cF^+,\cG) \to \sC \otimes_R \uHom(j^*\cF^+, j^*\cG) \to 0.
\]
By Step~1, the first chain complex is acyclic, so the map between the second and third is a quasi-isomorphism.

{\it Step 3. Let $\sD^+ \subset \sD(\cX,\bk)$ be the full triangulated subcategory generated by objects of the form $\cF^+$, and let $\iota: \sD^+ \hookrightarrow \sD(\cX,\bk)$ be the inclusion functor.  Then $j^* \circ \iota: \sD^+ \to \sD(\cU,\bk)$ is an equivalence of categories.}  Step~2 implies that $j^* \circ \iota$ is fully faithful.  Lemma~\ref{lem:open-gen} says that the image of $j^*$ generates $\sD(\cU,\bk)$.  But for any $\cF \in \sD(\cX,\bk)$, we clearly have
\begin{equation}\label{eqn:recolle1-open}
j^*\cF \cong j^*\cF^+ = j^*(\iota\cF^+),
\end{equation}
so the image of $j^*\circ \iota$ also generates $\sD(\cU,\bk)$.  We conclude that $j^*\circ \iota$ is essentially surjective, and hence an equivalence of categories.

We now define $j_!: \sD(\cU,\bk) \to \sD(\cX,\bk)$ to be the functor
\begin{equation}\label{eqn:j!-defn}
j_! = \iota \circ (j^* \circ \iota)^{-1}: \sD(\cU,\bk) \to \sD(\cX,\bk).
\end{equation}

{\it Step 4. The functor $j_!$ is left adjoint to $j^*$, and the adjunction map $\id \to j^*j_!$ is an isomorphism.}  By construction, there is a natural isomorphism $j^*j_! \cong \id$.  To show that $j_!$ is left adjoint to $j_*$, we must show that for all $\cF \in \sD(\cU,\bk)$ and $\cG \in \sD(\cX,\bk)$, the map
\begin{equation}\label{eqn:recolle2-adj}
\Hom(j_!\cF,\cG) \xrightarrow{j^*} \Hom(j^*j_!\cF,j^*\cG) \cong \Hom(\cF,j^*\cG)
\end{equation}
is an isomorphism.  By Lemma~\ref{lem:open-gen}, there exists an object $\tilde\cF \in \sD(\cX,\bk)$ such that $\cF \cong j^*\tilde\cF$.  By~\eqref{eqn:recolle1-open}, we have $\cF \cong j^*(\iota\tilde\cF^+)$, and then $j_!\cF = \iota (j^*\iota)^{-1}j^*(\iota\tilde\cF^+) \cong \iota\tilde\cF^+ \cong \tilde\cF^+$.  Step~2 then tells us that~\eqref{eqn:recolle2-adj} is an isomorphism.

{\it Step 5. There exists a distinguished triangle $j_!j^*\cF \to \cF \to i_*i^*\cF \to$, where the first two maps are adjunction maps.}  Consider the distinguished triangle
\[
\cF^+ \xrightarrow{\alpha} \cF \to i_*i^*\cF \to
\]
where the second map is from adjunction.  Step~4 implies that $\Hom(j_!j^*\cF, i_*i^*\cF) = \Hom(j_!j^*\cF, i_*i^*\cF[-1]) = 0$, so the adjunction map $\epsilon: j_!j^*\cF \to \cF$ factors uniquely through $\alpha$: there is a commutative diagram
\[
\begin{tikzcd}
j_!j^*\cF \ar[dr, "\epsilon"] \ar[d, "h"'] \\
\cF^+ \ar[r, "\alpha"] & \cF \ar[r] & i_*i^*\cF \ar[r] & {}
\end{tikzcd}
\]
To finish the proof, we must show that $h$ is an isomorphism.  Since $j_!j^*\cF$ and $\cF^+$ both lie in $\sD^+$, Step~3 tells us that it is enough to show that $j^*h$ is an isomorphism.  Since $j^*i_* = 0$, the bottom row above shows that $j^*\alpha$ is an isomorphism, and Step~4 tells us that $j^*\epsilon$ is an isomorphism, so $j^*h$ is an isomorphism, as desired.
\end{proof}

\begin{ex}\label{ex:gm-der!}
Let $\cX$ be as in Example~\ref{ex:gm-der}, and let $j: \cX_1 \hookrightarrow \cX$ be the inclusion of the open stratum.  Let $\cE_{\cX_1} = \ubk_{\cX_1}\{1\} = \cE_1|_{\cX_1}$.  Then $j_!\cE_{\cX_1}$ is the object described in Example~\ref{ex:gm-der}.
\end{ex}

\begin{lem}\label{lem:recolle-ind}
Let $\cY \subset \cX$ be a closed union of strata. Let $\cZ \subset \cY$ be a single closed stratum, and let $\cU = \cX \smallsetminus \cZ$.  Denote the inclusion maps as follows:
\[
\begin{tikzcd}
\cY \cap \cU \ar[r, "k_\cU"] \ar[d, "j_\cY"'] &
  \cU \ar[d, "j"] \\
\cY \ar[r, "k"] & \cX
\end{tikzcd}
\]
We then have
\[
j_! k_{\cU*} \cong k_* j_{\cY!}
\qquad\text{and}\qquad
j_* k_{\cU*} \cong k_* j_{\cY*}
\]
\end{lem}
Note that both $j$ and $j_\cY$ are inclusions of the complement of a single closed stratum, so the functors $j_!$, $j_*$, $j_{\cY!}$, and $j_{\cY*}$ exist by Lemma~\ref{lem:recollement1}.  On the other hand, $k$ and $k_\cU$ are closed inclusions, so $k_*$ and $k_{\cU*}$ are as in~\eqref{eqn:recolle-easy}.
\begin{proof}
Let $i_\cY: \cZ \hookrightarrow \cY$ and $i: \cZ \hookrightarrow \cX$ be the inclusion maps.  For $\cF \in \sD(\cY,\bk)$, consider the adjunction map $\cF \to i_{\cY*}i_\cY^*\cF$, and apply $k_*$.  Using properties of $i_{\cY*}$, $i_*$, and $k_*$ in the ordinary (unmixed) setting, this map can be identified with the adjunction map $(k_*\cF) \to i_*i^*(k_*\cF)$.  Therefore, $k_*$ commutes with the ``${}^+$'' construction from the proof of Lemma~\ref{lem:recollement1}.  

Let $\iota: \sD^+ \hookrightarrow \sD(\cX,\bk)$ and $\iota: \sD_\cY^+ \hookrightarrow \sD(\cY,\bk)$ be as in Step~3 of the proof of Lemma~\ref{lem:recollement1}.  The preceding paragraph implies that $k_*$ takes $\sD_\cY^+$ to $\sD^+$, and hence that it commutes with $\iota$.

Since we have $j^*k_* \cong k_{\cU*}j_\cY^*$ in the unmixed setting, the same isomorphism holds in the mixed setting as well.  Thus, the diagram
\[
\begin{tikzcd}
\sD_\cY^+ \ar[r, "j_\cY^* \circ \iota"] \ar[d, "k_*"'] &
  \sD(\cY \cap \cU,\bk) \ar[d, "k_{\cU*}"] \\
\sD^+ \ar[r, "j^* \circ \iota "] & \sD(\cU,\bk)
\end{tikzcd}
\]
commutes up to natural isomorphism.  The horizontal arrows are equivalences of categories, and we have
\[
j_! k_{\cU*} = \iota (j^*\iota)^{-1} k_{\cU*} \cong \iota k_* (j_\cY^*\iota)^{-1} \cong k_* j_{\cY!},
\]
as desired.  The proof that $j_* k_{\cU*} \cong k_* j_{\cY*}$ is similar.
\end{proof}

\begin{proof}[Proof of Theorem~\ref{thm:recollement}]
We proceed by induction on the number of strata in $\cZ$.  The case of a single stratum has been done in Lemma~\ref{lem:recollement1}.  Assume now that $\cZ$ has more than one stratum.  Choose a closed stratum $\cX_s \subset \cZ$, and let $\cX' = \cX \smallsetminus \cX_s$ and $\cZ' = \cX' \cap \cZ = \cZ \smallsetminus \cX_s$.  Let $j'$, $i'$, $j''$, $i''$, $j_\cZ$, and $i_\cZ$ be the inclusion maps indicated in the diagram below:
\[
\begin{tikzcd}
\cZ' = \cX' \cap \cZ \ar[r, "i'"] \ar[dd, "j_\cZ"'] &
  \cX' = \cX \smallsetminus \cX_s \ar[d, "j''"] &
  \cU \ar[l, "j'"'] \ar[dl, "j = j'' \circ j'"] \\
& \cX \\
\cZ \ar[ur, "i"]& \cX_s \ar[u, "i'' = i \circ i_\cZ"'] \ar[l, "i_\cZ"]
\end{tikzcd}
\]
By induction, $(j')^*$ and $(j'')^*$ both have left adjoints, denoted by $j'_!$ and $j''_!$, respectively.  Moreover, the adjunction maps $\eta': \id \to (j')^*j'_!$ and $\eta'': \id \to (j'')^*j''_!$ are isomorphisms.  It follows that $j^* \cong (j')^* \circ (j'')^*$ has a left adjoint, given by $j_! \cong j''_! \circ j'_!$.  The adjunction map $\id \to j^*j_!$ factors as
\[
\id \xrightarrow{\eta'} (j')^* j'_! \xrightarrow{(j')^*\eta'' j'_!} (j')^*(j'')^*j''_! j'_! \cong j^* j_!,
\]
so it is an isomorphism.

Next, let $\cG$ be the cone of the adjunction map $\epsilon: j_!j^*\cF \to \cF$.  By induction, we have distinguished triangles
\begin{gather*}
j''_!(j'')^*\cF \to \cF \to i''_*(i'')^*\cF \to, \\
j'_!(j')^*(j'')^*\cF \to (j'')^*\cF \to i'_*(i')^*(j'')^*\cF \to
\end{gather*}
where the first two maps in each line are adjunction maps.  Apply $j''_!$ to the latter, and combine it with the former into the following octahedral diagram:
\[
\begin{tikzcd}[column sep=0pt,row sep=tiny]
&&& {} \\
&& \hspace{-3em}j''_!i'_*(i')^*(j'')^*\cF \cong i_*j_{\cZ!}(i')^*(j'')^*\cF\hspace{-3em} \ar[dr] \ar[ur] &&& {} \\
& j''_!(j'')^*\cF \ar[dr] \ar[ur] && \cG \ar[ddr] \ar[urr] \\
&& \cF \ar[ur] \ar[drr] \\
j_!j^*\cF \cong j''_!j'_!(j')^*(j'')^*\cF\hspace{-3em} \ar[uur] \ar[urr] &&&& \hspace{-3em}i''_*(i'')^*\cF = i_*i_{\cZ*}(i'')^*\cF \ar[dr] \\
&&&&& {}
\end{tikzcd}
\]
Here, the topmost object has been rewritten using Lemma~\ref{lem:recolle-ind}.  Since $i_*$ is fully faithful, the rightmost distinguished triangle above shows that $\cG$ must lie in the image of $i_*$, say $\cG \cong i_*\cG'$.  We thus have a distinguished triangle
\[
j_!j^*\cF \to \cF \to i_*\cG' \to.
\]
Then~\cite[Corollaire~1.1.10]{bbd} tells us that $\cG'$ is unique up to canonical isomorphism, and~\cite[Proposition~1.1.9]{bbd} says that the assignment $\cF \mapsto \cG'$ is a functor.  Denote this functor by $i^*: \sD(\cX,\bk) \to \sD(\cZ,\bk)$.  It is then straightforward to check the remaining desired properties of $i^*$.

The construction of $j_*$ and $i^!$ is similar, and will be omitted.
\end{proof}

\begin{prop}\label{prop:recolle-commute}
Let $i: \cZ \hookrightarrow \cX$ be the inclusion of a closed union of strata, and let $j: \cU \hookrightarrow \cX$ be the complementary open inclusion.  Then the forgetful, monodromy, and coinvariant functors commute with all six functors in the recollement diagram.
\end{prop}
\begin{proof}
For the functors $i_*$ and $j^*$ from~\eqref{eqn:recolle-easy}, this is clear.

Let us now consider $j_!$.  From the construction in the proof of Theorem~\ref{thm:recollement}, it is clear that it is enough to prove the statement in the case where $\cZ$ consists of a single stratum.  In this case, the forgetful, monodromy, and coinvariants functors also commute with $i^*$ (which takes parity sheaves to parity sheaves), and hence with the ``${}^+$'' construction from the proof of Lemma~\ref{lem:recollement1}.  They then also commute with the functor $\iota$ from Step~3 of the proof of that lemma.  From~\eqref{eqn:j!-defn}, we conclude that they commute with $j_!$, as desired.

Next, consider $i^*$.  For $\cF \in \Dmix_\Gm(\cX,\bk)$, we can construct a commutative diagram
\[
\begin{tikzcd}
j_!j^*\For(\cF) \ar[r] \ar[d] & \For(\cF) \ar[r] \ar[d, equal] & i_*i^*\For(\cF) \ar[d, dashed] \ar[r] & {} \\
\For(j_!j^*\cF) \ar[r] & \For(\cF) \ar[r] & \For(i_*i^*\cF) \ar[r] & {}
\end{tikzcd}
\]
The first vertical arrow is an isomorphism by the previous paragraph, so the third one is as well.  Since $i_*$ is fully faithful, we conclude that $i^*\For(\cF) \cong \For(i^*\cF)$.

The same argument applies to $\Mon$ and $\Coi$, and the proofs for $j_*$ and $i^!$ are similar.
\end{proof}

\section{The perverse \texorpdfstring{$t$}{t}-structure}
\label{sec:perverse}

In this section, we will define the perverse $t$-structure on $\Dmix_\Gm(\cX,\bk)$ and on $\Dmix(\cX,\bk)$.  (We will not define any $t$-structure on $\Dmix_\mon(\cX,\bk)$.)  The case of $\Dmix_\Gm(\cX,\bk)$ is rather similar to~\cite[\S3.5]{ar:mpsfv2} (although that paper assumes that each stratum is an affine space).

Let $\varpi$ be a generator of the unique maximal ideal of $\bk$.  (If $\bk$ is a field, then $\varpi = 0$.)  For any stratum $\cX_s$, we have a parity sheaf $\ubk_{\cX_s}\{\dim X_s\}$. Define an object $\ubkpi_{\cX_s}\{\dim X_s\}$ of $\Dmix_\Gm(\cX_s,\bk)$ or $\Dmix(\cX_s,\bk)$ as follows:
\[
\ubkpi_{\cX_s}\{\dim X_s\} =
\begin{cases}
\cone(\ubk_{\cX_s}\{\dim X_s\} \xrightarrow{\varpi \cdot \id} \ubk_{\cX_s}\{\dim X_s\}) & \text{if $\bk$ is not a field,} \\
\ubk_{\cX_s}\{\dim X_s\} & \text{if $\bk$ is a field.}
\end{cases}
\]

\begin{lem}\label{lem:strat-tstruc}
Let $\cA_\Gm \subset \Dmix_\Gm(\cX_s,\bk)$, resp.~$\cA \subset \Dmix(\cX_s,\bk)$ be the full subcategory generated under extensions by the objects
\begin{equation}\label{eqn:strat-tgen}
\ubk_{\cX_s}\{\dim X_s\}\la n\ra
\qquad\text{and}\qquad
\ubkpi_{\cX_s}\{\dim X_s\}\la n\ra
\end{equation}
for $n \in \Z$.  Then $\cA_\Gm$, resp.~$\cA$, is the heart of a unique bounded $t$-structure on $\Dmix_\Gm(\cX_s,\bk)$, resp.~$\Dmix(\cX_s,\bk)$.
\end{lem}
\begin{proof}
For brevity, let $\uubk_s = \ubk_{\cX_s}\{\dim X_s\}$.  Let $\sD$ denote either $\Dmix_\Gm(\cX_s,\bk)$ or $\Dmix(\cX_s,\bk)$.  Consider the following two claims:
\begin{enumerate}
\item $\sD$ is generated as a triangulated category by objects of the form $\uubk_s\la n\ra$.\label{it:strat-gen}
\item We have\label{it:strat-hom}
\[
\Hom(\uubk_s, \uubk_s[m]\la n\ra)
= \begin{cases}
0 & \text{if $m < 0$, or if $m = 0$ and $n \ne 0$,} \\
\bk & \text{if $m = n = 0$,} \\
\text{a free $\bk$-module} & \text{if $m = 1$.}
\end{cases}
\]
\end{enumerate}
Claim~\eqref{it:strat-gen} is true by Lemma~\ref{lem:con-gen}.  Claim~\eqref{it:strat-hom} can be checked by direct calculations in the dgg rings $\uEnd(\uubk_s) \cong \coh^\bullet_\Gm(\cX_s,\bk)$ and $\sA \otimes_R \uEnd(\uubk_s) \cong \sA \otimes_R \coh^\bullet_\Gm(\cX_s,\bk)$, recalling that $\coh^\bullet_\Gm(\cX_s,\bk)$ is ``concentrated on the diagonal.''  (Note that in $\Dmix_\Gm(\cX_s,\bk)$, we actually have $\Hom(\uubk_s, \uubk_s[1]\la n\ra) = 0$ for all $n$.  In $\Dmix(\cX_s,\bk)$, the same vanishing holds for $n \ne -2$, but $\Hom(\uubk_s, \uubk_s[1]\la -2\ra)$ may be nonzero.  However, it is still a submodule of the free $\bk$-module $(\sA \otimes_R \coh^\bullet_\Gm(\cX_s,\bk))^1_2$, so it is free over $\bk$.)

Claims~\eqref{it:strat-gen} and~\eqref{it:strat-hom} are precisely the hypotheses of~\cite[Lemma~A.1]{ahr:ies}, which asserts the existence of the desired $t$-structure.
\end{proof}

The $t$-structure constructed in Lemma~\ref{lem:strat-tstruc} will be denoted by
\[
(\p\Dmix_\Gm(\cX_s,\bk)^{\le 0}, \p\Dmix_\Gm(\cX_s,\bk)^{\ge 0}),
\quad\text{resp.}\quad
(\p\Dmix(\cX_s,\bk)^{\le 0}, \p\Dmix(\cX_s,\bk)^{\ge 0}).
\]
In the case of $\Dmix_\Gm(\cX_s,\bk)$, by~\cite[Remark~A.2]{ahr:ies}, the heart of this $t$-structure is equivalent to the category of finitely generated graded $\bk$-modules.

\begin{defn}
For each stratum $\cX_s$, let $j_s: \cX_s \hookrightarrow \cX$ be the inclusion map. The \emph{perverse $t$-structure} on $\Dmix_\Gm(\cX,\bk)$ is the $t$-structure given by
\begin{align*}
\p\Dmix_\Gm(\cX,\bk)^{\le 0} &= \{ \cF \in \Dmix_\Gm(\cX,\bk) \mid \text{$j_s^*\cF \in \p\Dmix_\Gm(\cX_s,\bk)^{\le 0}$ for all $s$} \}, \\
\p\Dmix_\Gm(\cX,\bk)^{\ge 0} &= \{ \cF \in \Dmix_\Gm(\cX,\bk) \mid \text{$j_s^!\cF \in \p\Dmix_\Gm(\cX_s,\bk)^{\ge 0}$ for all $s$} \}.
\end{align*}
The perverse $t$-structure on $\Dmix(\cX,\bk)$ is defined similarly.  
\end{defn}

The fact that these are indeed $t$-structures follows from standard properties of the recollement formalism.

\begin{lem}
The functor $\For: \Dmix_\Gm(\cX,\bk) \to \Dmix(\cX,\bk)$ is $t$-exact.
\end{lem}
\begin{proof}
If $\cX$ consists of a single stratum, this is immediate from Lemma~\ref{lem:strat-tstruc}.  In the general case, it follows from Proposition~\ref{prop:recolle-commute}.
\end{proof}

\begin{prop}\label{prop:eqvt-ff}
The functor $\For: \Perv_\Gm(\cX,\bk) \to \Perv(\cX,\bk)$ is fully faithful.
\end{prop}
\begin{proof}
Note that there is a short exact sequence of $R$-modules
\begin{equation}\label{eqn:eqvt-ses1}
0 \to R \to \sA \to R\la 2\ra[-1] \to 0.
\end{equation}
If we equip $\sA$ with the differential $\kappa$, and $R$ and $R\la 2\ra[-1]$ with the zero differential, then this is actually a short exact sequence of chain complexes, or of dgg $R$-modules.

Now let $\cF, \cG \in \Dmix_\Gm(\cX,\bk)$. All three terms in~\eqref{eqn:eqvt-ses1} are flat over $R$, so if we apply $({-}) \otimes_R \uHom(\cF,\cG)$, we get a short exact sequence of chain complexes
\[
0 \to \uHom(\cF,\cG) \xrightarrow{\For} \sA \otimes_R \uHom(\cF,\cG) \to \uHom(\cF,\cG)\la 2\ra[-1] \to 0.
\]
Now take the long exact sequence in cohomology.  Part of this sequence is
\begin{multline*}
\cdots \to \Hom(\cF,\cG\la 2\ra[-2]) \to
\Hom(\cF,\cG) \to \Hom(\For(\cF),\For(\cG)) \\
\to \Hom(\cF,\cG\la 2\ra[-1]) \to \cdots.
\end{multline*}
If $\cF$ and $\cG$ are perverse, the first and last terms above vanish. We conclude that $\Hom(\cF,\cG) \to \Hom(\For(\cF),\For(\cG))$ is an isomorphism.
\end{proof}

\section{Constructibility}
\label{sec:rfree-constr}

We have seen in Example~\ref{ex:mon-der1} that $\Dmix_\mon(\cX,\bk)$ may contain objects ``of infinite type,'' i.e., for which certain graded $\Hom$-spaces may not be finitely generated over $\bk$.  We will see more such examples in~\S\ref{sec:jordan}, in the form of infinite-rank pro-unipotent local systems.  To distinguish these ``large'' objects from more manageable ones, we make the following definition.

\begin{defn}
An object $\cF \in \Dmix_\mon(\cX,\bk)$ is said to be \emph{constructible} if it lies in the essential image of $\Mon: \Dmix(\cX,\bk) \to \Dmix_\mon(\cX,\bk)$. 
\end{defn}

\begin{lem}\label{lem:constr-recolle}
Let $i: \cZ \hookrightarrow \cX$ be the inclusion of a closed union of strata, and let $j: \cU \hookrightarrow \cX$ be the complementary open inclusion.  For $\cF \in \Dmix_\mon(\cX,\bk)$, the following conditions are equivalent:
\begin{enumerate}
\item $\cF$ is constructible.\label{it:constr-tot}
\item $j^*\cF$ and $i^*\cF$ are constructible.\label{it:constr-*}
\item $j^*\cF$ and $i^!\cF$ are constructible.\label{it:constr-!}
\end{enumerate}
\end{lem}
\begin{proof}
Proposition~\ref{prop:recolle-commute} tells us that condition~\eqref{it:constr-tot} implies the other two conditions.

Let us now show that condition~\eqref{it:constr-*} implies condition~\eqref{it:constr-tot}.  If $j^*\cF$ and $i^*\cF$ are constructible, Proposition~\ref{prop:recolle-commute} implies that $j_!j^*\cF$ and $i_*i^*\cF$ are as well.  That is, the first and third terms of the distinguished triangle
\[
j_!j^*\cF \to \cF \to i_*i^*\cF \to
\]
lie in the essential image of $\Mon$.  Since $\Mon$ is fully faithful, the connecting morphism $i_*i^*\cF \to j_!j^*\cF[1]$ is in the image of $\Mon$, and hence so is its cocone $\cF$.

The proof that condition~\eqref{it:constr-!} implies condition~\eqref{it:constr-tot} is similar.
\end{proof}

The main result of this section is that for an $R$-free variety, all monodromic complexes are constructible.  Most of the work is spent on the case of a single stratum.

\begin{prop}\label{prop:constr-1strat}
Suppose $\cX_s$ is an $R$-free stratum. Then $\Dmix_\mon(\cX,\bk)$ is generated as a triangulated category by objects of the form $\Mon(\cE_s)\la n\ra$.
\end{prop}
\begin{proof}
In the case where $\bk$ is not a field, let $\varpi$ be a generator of its unique maximal ideal.  The text of the proof below is adapted to this case.  The case where $\bk$ is a field is slightly easier.  To obtain the proof in the field case, read the argument below with the understanding that $\varpi = 0$.  (There are some additional comments on the field case after the end of the proof.)

We introduce some additional notation related to bigraded $\bk$-modules.  For a homogeneous element $m$ in a bigraded $\bk$-module $M$ of bidegree $\binom{e}{f}$, recall that the cohomological degree is given by $|m| = e$.  We define its \emph{total degree} by $\tot m = e - f$.

Let $\sH_s = \coh^\bullet_\Gm(\cX_s,\bk)$, and let $\Free(\sH_s)$ denote the category of finitely generated bigraded free $\sH_s$-modules. It is easy to see that the functor $\uHom(\cE_s,{-})$ gives rise to an equivalence of categories
\[
\uHom(\cE_s,{-}): \PargrGm{\cX_s} \simto \Free(\sH_s).
\]
For the rest of this section, we identify these categories.  Thus, the categories $\Dmix_\Gm(\cX_s,\bk)$, $\Dmix(\cX_s,\bk)$, and $\Dmix_\mon(\cX_s,\bk)$ all consist of objects of $\Free(\sH_s)$ with additional data.  In this language, an object of $\Dmix_\mon(\cX,\bk)$ is a pair $(M,\delta)$, where $M \in \Free(\sH_s)$, and $\delta: R^\vee \otimes M \to R^\vee \otimes M[1]$ is an $R^\vee \otimes \sH_s$-module homomorphism satisfying $\delta^2 = \Theta \cdot \id$.

Given such an object $(M,\delta)$, we wish to prove that it lies in the subcategory generated by objects of the form
\[
\Mon(\cE_s)\la n\ra =
\left(
\sH_s\la n\ra \oplus \sH_s\la n -2\ra[1], [\begin{smallmatrix} & \sr \\ \xi & \end{smallmatrix}]\right).
\]
We proceed by induction on the rank of $M$ (as a free $\sH_s$-module).  Of course, if $M$ has rank $0$, there is nothing to prove.

If $M$ has rank${}> 0$, the argument is lengthy but elementary.  We will set up quite a lot of notation related to a basis for $M$.  We will then consider two different cases involving the behavior of $\delta$ in this basis.

{\em Step 1. Set-up and notation.}
Choose a homogeneous $\sH_s$-basis $a_1, \ldots, a_n$ for $M$, and assume that $a_1$ has maximal total degree among these basis elements.  Let $t = \tot a_1$ and $b = |a_1|$. Write the differential of $a_1$ in terms of this basis as
\begin{equation}\label{eqn:constr-rev0}
\delta(a_1) = \sum_{i=1}^n c_i a_i
\quad\text{where}\quad
c_i \in R^\vee \otimes \sH_s
\quad\text{and}\quad
\left\{
\begin{aligned}
|c_i| + |a_i| &= b + 1, \\
\tot c_i + \tot a_i &= t + 1.
\end{aligned}
\right.
\end{equation}
Now, all elements of $R^\vee \otimes \sH_s$ have even, nonnegative cohomological and total degrees.  Since $t \ge \tot a_i$ for all $i$, we see that in any term with $c_i \ne 0$, we must have $\tot c_i \ge 2$, and hence
\[
\begin{aligned}
|a_i| &\le b + 1\\
|a_i| &\equiv b + 1 \pmod 2
\end{aligned}
\qquad\text{and}\qquad
\begin{aligned}
\tot a_i &\le t - 1 \\
\tot a_i &\equiv t -1 \pmod 2.
\end{aligned}
\]
Assume that the basis elements meeting these degree conditions are $a_2, \ldots, a_d$. (Of course, $a_1$ cannot satisfy these conditions.)  

Since $\sH_s$ is concentrated in total degree $0$, we see that any element $c \in R^\vee \otimes \sH_s$ with positive total degree must be divisible by $\sr$.  This applies to every nonzero coefficient in~\eqref{eqn:constr-rev0}.  Factor these coefficients as
\[
c_i = \sr c'_i,
\]
where $c'_i \in R^\vee \otimes \sH_s$ satisfies $|c'_i| = |c_i|$ and $\tot c'_i = \tot c_i - 2$.

Break up~\eqref{eqn:constr-rev0} as
\[
\delta(a_1) = \sr D_1 + \sr D_2 + \sr D_3 + \sr D_4
\]
where
\[
D_1 = \sneg\sum_{\substack{2 \le i \le d\\ |a_i| = b+1 \\ \tot a_i = t - 1}}\sneg c'_i a_i
\qquad
D_2 = \sneg\sum_{\substack{2 \le i \le d\\ |a_i| = b-1 \\ \tot a_i = t - 1}}\sneg c'_i a_i
\qquad
D_3 = \sneg\sum_{\substack{2 \le i \le d\\ |a_i| \le b-3 \\ \tot a_i = t - 1}}\sneg  c'_i a_i
\qquad
D_4 = \sneg\sum_{\substack{2 \le i \le d\\ \tot a_i \le t - 3}}\sneg c'_i a_i
\]
The coefficients $c'_i$ in $D_4$ still have positive total degree, so we have
\[
D_4 \in \sr(R^\vee \otimes M).
\]
On the other hand, every nonzero coefficient $c'_i$ ocurring in $D_1$, $D_2$, or $D_3$ has total degree $0$, and thus lies in $\sH_s$.  Since $\sH_s$ is free as an $R$-module, we can decompose it as an $R$-module as $\sH_s \cong R \oplus C$.  Let
\[
\fa = C \oplus \xi^2R.
\]
Note that as a $\bk$-module, we have $\sH_s \cong \bk \oplus \bk\xi \oplus \fa$.  It is easy to see from this that $\fa$ is actually an ideal in $\sH_s$.  All homogeneous elements of cohomological degree${}\ge 4$ lie in $\fa$.  The coefficients in $D_3$ have cohomological degree${}\ge 4$, so
\[
D_3 \in \fa(R^\vee \otimes M).
\]
In $D_2$, the coefficients $c'_i$ belong to $\sH^2_s = \coh^2_\Gm(\cX_s,\bk)$, which is homogeneous of bidegree $\binom{2}{2}$.  We have $\sH^2_s = \bk\xi \oplus (\fa \cap \sH^2_s)$, so we can rewrite $D_2$ as
\[
D_2 = D_2' + D_2''
\]
where
\[
D_2' = \sneg\sum_{\substack{2 \le i \le d\\ |a_i| = b-1 \\ \tot a_i = t - 1}}\sneg b'_i\xi a_i,
\quad b'_i \in \bk
\qquad\text{and}\qquad
D_2'' = \sneg\sum_{\substack{2 \le i \le d\\ |a_i| = b-1 \\ \tot a_i = t - 1}}\sneg b''_i a_i,
\quad b''_i \in \fa \cap \sH^2_s.
\]
In particular, we have
\[
D_2'' \in \fa(R^\vee \otimes M).
\]
Lastly, in $D_1$, the coefficients $c'_i \in \sH_s$ have cohomological degree $0$, so they belong to $\bk$.  We are now done setting up the needed notation.

{\it Step 2. Assume that some coefficient in $D_1$ is a unit in $\bk$.}
Without loss of generality, suppose $c'_2$ is a unit.  Let
\[
a'_2 = \sum_{i=2}^d c'_i a_i = \sr^{-1}\delta(a_1).
\]
Because $c'_2$ is a unit in $\bk$, the set $\{a_1, a_2', a_3, \ldots, a_n \}$ is still a basis for $M$.  In this new basis, we have $\delta(a_1) = \sr a_2'$, and hence $\delta^2(a_1) = \sr\xi a_1 = \sr \delta(a_2')$.  Since multiplication by $\sr$ is an injective map, it follows that $\delta(a_2') = \xi a_1$.

Let $M' = \spn_{\sH_s} \{a_3, a_4, \ldots, a_n \}$, and decompose $M$ as
\[
M = \sH_s a_2' \oplus \sH_s a_1 \oplus M'.
\]
The observations above show that with respect to this decomposition, the differential $\delta$ can be written in the form
\[
\delta =
\begin{bmatrix}
& \sr & f_2' \\
\xi & & f_1 \\
& & \delta'
\end{bmatrix}
\]
for suitable maps $f_2'$, $f_1$, and $\delta'$.  We see that $(M,\delta)$ is the cone of the map
\[
f = \begin{bmatrix} f_2' \\ f_1 \end{bmatrix}:
(M'[-1], -\delta') \to \Mon(\cE_s[-b-1]\la b - t + 2\ra).
\]
Since $M'$ has smaller rank than $M$ as an $\sH_s$-module, we are done.

{\it Step 3. Assume that no coefficient in $D_1$ is a unit in $\bk$.}
In this case, every coefficient $c'_i$ in $D_1$ belongs to the maximal ideal of $\bk$, and hence
\[
D_1 \in \varpi(R^\vee \otimes M).
\]
Now let
\[
a_2' = \xi^{-1}D_2' = \sneg \sum_{\substack{2 \le i \le d\\ |a_i| = b-1 \\ \tot a_i = t - 1}}\sneg b'_i a_i
\qquad\text{and}\qquad
a'_1 = \delta(a_2').
\]
so that $\xi a'_1 = \delta(D_2')$.  Expand $a'_1$ in our basis as
\[
a'_1 = \sum_{i=1}^n e_i a_i
\qquad\text{with}\qquad
e_i \in R^\vee \otimes \sH_s.
\]
Note that $\deg a'_1 = \deg a_1$, and hence that $\deg e_1 = \binom{0}{0}$.  In other words, $e_1$ is an element of $\bk$.  We claim that it is invertible.  If it were not invertible, we would have
\[
a'_1 \in \varpi(R^\vee \otimes M) + \spn_{R^\vee \otimes \sH_s}\{a_2,\ldots, a_n\},
\]
and hence that
\begin{align*}
\delta^2(a_1) &= \sr(\delta(D_1) + \delta(D_2') + \delta(D_2'') + \delta(D_3) + \delta(D_4)) \\
&= \sr(\delta(D_1) + \xi a'_1 + \delta(D_2'') + \delta(D_3) + \delta(D_4)) \\
&\in \sr \varpi (R^\vee \otimes M) + \sr \fa (R^\vee \otimes M) + \sr^2 (R^\vee \otimes M) + \spn_{R^\vee \otimes \sH_s}\{a_2,\ldots, a_n\}.
\end{align*}
But $\delta^2(a_1) = \sr\xi a_1$ does not belong to the right-hand side, and we have contradiction.  Thus, $e_1$ is a unit.  It follows that $\{a'_1, a_2, \ldots, a_n\}$ is still a basis for $M$.

Next, we claim that some coefficient $b'_i$ in $a_2'$ is a unit in $\bk$.  If not, we would have $a_2' \in \varpi M$, and hence $a'_1 \in \varpi M$ as well, contradicting the fact that $a'_1$ is part of a basis for $M$.  Assume without loss of generality that $a_2$ occurs in $a_2'$, and that $b'_2$ is a unit. We conclude that $\{a'_1, a_2', a_3, \ldots, a_n \}$ is still a basis for $M$.

Let $M' = \spn_{\sH_s} \{a_3, a_4, \ldots, a_n \}$, and decompose $M$ as
\[
M = \sH_s a_1' \oplus \sH_s a_2' \oplus M'.
\]
Note that $\delta(a_1') = \delta^2(a_2') = \sr\xi a_2'$.  Thus, with respect to this decomposition, the differential $\delta$ can be written in the form
\[
\delta =
\begin{bmatrix}
& 1 & g_1' \\
\sr\xi & & g_2' \\
& & \delta'
\end{bmatrix}
\]
Let
\[
(M_0,\delta_0) =
\left(
\sH_s[-b]\la b-t\ra \oplus \sH_s[-b+1]\la b-t\ra, 
\begin{bmatrix}
& 1 \\ \sr\xi & 
\end{bmatrix}
\right).
\]
As in the previous case, we see that $(M,\delta)$ is the cone of the map
\[
g = \begin{bmatrix} g_1' \\ g_2' \end{bmatrix}:
(M'[-1], -\delta') \to (M_0,\delta_0).
\]
We claim that $(M_0,\delta_0)$ is homotopic to $0$.  Indeed, we have
\[
\id_{M_0} = \delta_0 h + h \delta_0
\qquad\text{where}\qquad
h = 
\begin{bmatrix}
0 & 0 \\ 1 & 0 
\end{bmatrix}.
\]
We conclude that $(M,\delta)$ is isomorphic in $\Dmix_\mon(\cX_s,\bk)$ to $(M', \delta')$.  Since $M'$ has smaller rank than $M$ as an $\sH_s$-module, we are done.
\end{proof} 

In the case where $\bk$ is a field, the headings of Steps~2 and~3 of the preceeding proof can be simplified to ``Assume that $D_1 \ne 0$'' and ``Assume that $D_1 = 0$,'' respectively, and the second paragraph of Step~3 discusses the claim that $a_2' \ne 0$.

\begin{thm}\label{thm:rfree-mon}
If $\cX$ is $R$-free, then the monodromy functor
\[
\Mon: \Dmix(\cX,\bk) \to \Dmix_\mon(\cX,\bk)
\]
is an equivalence of categories.
\end{thm}
\begin{proof}
Since $\Mon$ is fully faithful, we just need to show that it is essentially surjective. We proceed by induction on the number of strata.  If $\cX$ consists of a single stratum, this follows from Proposition~\ref{prop:constr-1strat}.  Otherwise, let $i: \cX_s \hookrightarrow \cX$ be the inclusion of a closed stratum, and let $j: \cX \smallsetminus \cX_s \hookrightarrow \cX$ be the inclusion of its complement.  Consider the distinguished triangle
\[
j_!j^*\cF \to \cF \to i_*i^*\cF \to.
\]
By induction, $j^*\cF$ and $i^*\cF$ both lie in the image of $\Mon$, and then by Lemma~\ref{lem:constr-recolle}, so does $\cF$.
\end{proof}

\section{Jordan block local systems}
\label{sec:jordan}

\subsection{Motivation: the punctured plane}

The set-up of the paper includes the assumption that every $\Gm$-equivariant local system on each $\cX_s$ is trivial.  But if we drop the $\Gm$-equivariance, there may well be nontrivial local systems.  A motivating example is the case $\cX = \bA^1 \smallsetminus \{0\}$, equipped with the standard action of $\Gm$. The constant sheaf $\ubk_{\bA^1 \smallsetminus \{0\}}$ admits no $\Gm$-equivariant self-extensions, but if we drop the $\Gm$-equivariance, then it does admit self-extensions.  The local systems obtained by repeated self-extensions of the constant sheaf are those with unipotent monodromy.

The analogue of this phenomenon in the mixed modular setting is based on the fact that 
\begin{equation}\label{eqn:a1-ext1}
\Hom_{\Dmix(\bA^1 \smallsetminus \{0\},\bk)}(\ubk,\ubk\la-2\ra[1]) \cong \bk.
\end{equation}
It is possible to prove this directly in $\Dmix(\bA^1 \smallsetminus \{0\},\bk)$.  Below, we will see how to deduce this from a computation in $\Dmix_\mon(\bA^1 \smallsetminus \{0\},\bk)$.

Since $\coh^\bullet_\Gm(\bA^1 \smallsetminus \{0\},\bk) = \bk$, the category $\PargrGm{\bA^1 \smallsetminus \{0\}}$ is equivalent to the category of finite-rank bigraded free $\bk$-modules.  For any object $\cF \in \PargrGm{\bA^1 \smallsetminus \{0\}}$, we have $\Theta \cdot \id_\cF = 0$.  In this language, an object of $\DGmon{\bA^1 \smallsetminus \{0\}}$ is a pair $(M,\delta)$, where $M$ is a finite-rank free bigraded $\bk$-module, and $\delta : R^\vee \otimes M \to R^\vee \otimes M[1]$ is an $R^\vee$-linear map such that $\delta^2 = 0$.

In other words, $\DGmon{\bA^1 \smallsetminus \{0\}}$ is simply the dgg category whose objects are bounded chain complexes of finitely generated free $R^\vee$-modules.  Let $R^\vee\lgmod$ denote the abelian category of finitely generated graded $R^\vee$-modules.  We can then identify
\begin{equation}\label{eqn:mon-a1}
\Dmix_\mon(\bA^1 \smallsetminus \{0\},\bk) \cong
\Db(R^\vee\lgmod).
\end{equation}
Under this equivalence, the free module $R^\vee$ corresponds to the shifted constant sheaf $\ubk\{1\}$ on $\bA^1 \smallsetminus \{0\}$.  For brevity, let us denote this parity sheaf by
\[
\uubk = \ubk_{\bA^1 \smallsetminus \{0\}}\{1\}.
\]

Let us write down some specific objects $\Dmix_\mon(\bA^1 \smallsetminus \{0\},\bk)$.  For $n \ge 1$, let $\cJ_n(\uubk) \in \Dmix_\mon(\bA^1 \smallsetminus \{0\},\bk)$ be the object given by
\[
\cJ_n(\uubk) =
\begin{tikzcd}[column sep=large]
\uubk\la -2n\ra[1] \ar[r, "{\sst[1]}" description, "\sr^n \cdot \id" above=2] &
\uubk
\end{tikzcd}
.
\]
Note that $\cJ_1(\uubk) \cong \Mon(\uubk)$.  Under~\eqref{eqn:mon-a1}, the chain complex $\cJ_n(\uubk)$ becomes a free resolution of the $R^\vee$-module $\bk[\sr]/(\sr^n)$.  For every $n \ge 1$, there is a nonsplit distinguished triangle
\[
\cJ_n(\uubk)\la -2\ra \xrightarrow{\sr} \cJ_{n+1}(\uubk) \to \Mon(\uubk) \to,
\]
corresponding to the short exact sequence of $R^\vee$-modules $0 \to \bk[\sr]/(\sr^n)\la -2\ra \to \bk[\sr]/(\sr^{n+1}) \to \bk \to 0$.  In particular, when $n = 1$, we get a nonsplit distinguished triangle
\[
\Mon(\uubk)\la -2\ra \xrightarrow{\sr} \cJ_2(\uubk) \to \Mon(\uubk) \to.
\]
This shows that $\Hom(\Mon(\uubk), \Mon(\uubk)\la -2\ra[1])$ is at least nonzero.  To see that it is isomorphic to $\bk$ (as was claimed in~\eqref{eqn:a1-ext1}), observe that $\Ext^1_{R^\vee}(\bk, \bk\la -2\ra) \cong \bk$.
  
Morally, $\cJ_n(\uubk)$ should be thought of as a local system whose monodromy is a unipotent Jordan block of size $n$.  These objects clearly lie in the image of $\Mon$, so ``unipotent Jordan block local systems'' exist in $\Dmix(\bA^1 \smallsetminus \{0\},\bk)$ as well.

But in $\Dmix_\mon(\bA^1 \smallsetminus \{0\},\bk)$, it also makes to consider the object
\[
\cJ(\uubk)
\]
whose underlying graded parity sheaf is just $\uubk$, equipped with the zero differential.  The graded endomorphism ring of this object is infinitely generated over $\bk$: we have
\[
\bigoplus_{n \in \Z} \Hom(\cJ(\uubk), \cJ(\uubk)\la n\ra) \cong \uEnd(\cJ(\uubk)) \cong R^\vee.
\]
Remark~\ref{rmk:hom-nilp} implies that $\cJ(\uubk)$ does not lie in the image of $\Mon$.

Morally, $\cJ(\uubk)$ should be thought of as an infinite-rank local system; it is the pro-unipotent projective cover of the constant sheaf $\Mon(\uubk)$ in the category of local systems.

\subsection{Jordan block functors}

We will now generalize the construction of the preceding subsection to an arbitrary $R$-trivial variety.  As above, for any parity sheaf $\cF$ on an $R$-trivial variety, we have $\Theta \cdot \id_\cF = 0$.  Thus, objects of $\DGmon{\cX}$ or $\Dmix_\mon(\cX,\bk)$ are pairs $(\cF, \delta)$ where $\delta \in \uEnd_\sS(\cF)$ satisfies $\delta^2 = 0$.

Suppose now that $(\cF,\delta)$ is an object of $\DGeq{\cX}$.  The obvious map $\uEnd(\cF) \to \uEnd_\sS(\cF)$ induced by $\iota_{R^\vee}: \bk \to R^\vee$ lets us regard $(\cF,\delta)$ as an object of $\DGmon{\cX}$.  We have just defined a functor $\cJ: \DGeq{\cX} \to \DGmon{\cX}$.  Passing to homotopy categories, we obtain a functor
\[
\cJ: \Dmix_\Gm(\cX,\bk) \to \Dmix_\mon(\cX,\bk).
\]

Next, for any $n \ge 1$, define $\cJ_n:  \DGeq{\cX} \to \DGmon{\cX}$ by
\[
\cJ_n(\cF,\delta) = 
\begin{tikzcd}[column sep=large]
\cF\la -2n\ra[1] \ar[loop above, out=120, in=60, distance=20, "{\sst[1]}"{description, pos=0.18}, "-\delta"]
  \ar[r, "{\sst[1]}" description, "\sr^n" above=2] &
\cF \ar[loop above, out=120, in=60, distance=20, "{\sst[1]}"{description, pos=0.18}, "\delta"]
\end{tikzcd}
.
\]
Again, passing to homotopy categories, we obtain a functor 
\[
\cJ_n: \Dmix_\Gm(\cX,\bk) \to \Dmix_\mon(\cX,\bk).
\]
Comparing with~\eqref{eqn:mon-concrete}, we see that there is a natural isomorphism
\[
\cJ_1(\cF) \cong \Mon(\For(\cF)).
\]

Informally, $\cJ_n$ should be thought of as ``tensoring with a unipotent Jordan block local system of rank $n$,'' and $\cJ$ should be thought of as ``tensoring with the pro-unipotent projective cover of the constant sheaf.''

\begin{prop}\label{prop:jordan}
Assume that $\cX$ is $R$-trivial, and let $\cF \in \Dmix_\Gm(\cX,\bk)$.
\begin{enumerate}
\item There are natural isomorphisms
\[
\Coi(\cJ(\cF)) \cong \cF
\qquad\text{and}\qquad
\cJ_1(\cF) \cong \Mon(\For(\cF)).
\]
\item There are natural isomorphisms
\[
\cJ(\D\cF) \cong (\D(\cJ\cF))\la 2\ra[-1]
\qquad\text{and}\qquad
\cJ_n(\D\cF) \cong (\D(\cJ_n\cF))\la 2-2n\ra.
\]
\item For any $n \ge 1$, there is a functorial distinguished triangle
\[
\cJ(\cF)\la -2n\ra \xrightarrow{\Nilp_{\cJ(\cF)}^n} \cJ(\cF) \to \cJ_n(\cF) \to.
\]
\item For any $n, m \ge 1$, there is a functorial distinguished triangle
\[
\cJ_n(\cF)\la -2m\ra \to \cJ_{n+m}(\cF) \to \cJ_m(\cF) \to.
\]
\end{enumerate}
\end{prop}

These statements are all immediate consequences of the definitions.  The details are left to the reader.

\section{The nearby cycles formalism}
\label{sec:nearby}

In this section, we define the nearby cycles functor and discuss its main properties.  We also discuss two additional related functors, the maximal extension functor and the vanishing cycles functor, following the organization of ideas in~\cite{beilinson, reich}.  

\subsection{The nearby cycles functor}

Let $\Gm$ act on $\bA^1$ by a nontrivial character $z \mapsto z^c$.  Assume that the integer $c$ is invertible in $\bk$.  This implies that the equivariant cohomology $\coh^\bullet_\Gm(\bA^1 \smallsetminus \{0\}, \bk)$ is trivial.

With respect to this action, suppose that we have a $\Gm$-equivariant map
\[
f: X \to \bA^1.
\]
Let $X_\gen = f^{-1}(\bA^1 \smallsetminus 0)$, and let $X_0 = f^{-1}(0)$.  As usual, we use the stacky notation $\cX_0 = X_0/H$ and $\cX_\gen = X_\gen/H$. Let
\[
\bi: \cX_0 \hookrightarrow \cX
\qquad\text{and}\qquad
\bj: \cX_\gen \hookrightarrow \cX
\]
be the inclusion maps.  We add the following assumptions:
\begin{enumerate}
\item Each stratum $\cX_s$ is contained in either $\cX_0$ or $\cX_\gen$.
\item The variety $\cX_0$ is $R$-free.
\end{enumerate}
We also have:
\begin{enumerate}
\addtocounter{enumi}{2}
\item The variety $\cX_\gen$ is $R$-trivial.
\end{enumerate}
This is not an extra assumption, but rather a consequence of the conditions above.  Indeed, for parity sheaves $\cF$ and $\cG$ on $X_\gen$, the action of $R$ on $\uHom(\cF,\cG)$ factors through $\coh^\bullet_\Gm(\bA^1 \smallsetminus \{0\}) \cong \bk$.

We define the \emph{nearby cycles functor}
\[
\Psi_f: \Dmix_\Gm(\cX_\gen,\bk) \to \Dmix(\cX_0,\bk)
\]
by the formula
\[
\Psi_f(\cF) = \Mon^{-1}\bi^*\bj_*\cJ(\cF)\la -2\ra.
\]
Note that Theorem~\ref{thm:rfree-mon} tells us that $\bi^*\bj_*\cJ(\cF) \in \Dmix_\mon(\cX_0,\bk)$ is automatically constructible, and that it makes sense to apply the inverse functor of $\Mon$.

The Tate twist $\la -2\ra$ in the formula above is included for compatibility with the literature.  To elaborate, when $\bk$ is a field of characteristic $0$, this Tate twist corresponds to the operation that is usually denoted by $(1)$ in the context of mixed $\ell$-adic sheaves or mixed Hodge modules (see~\cite[\S6.1]{ar:kdsf}).  The Tate twist above thus matches that in~\eqref{eqn:nearby-analytic}, and the Tate twists in a number of statements below (such as Propositions~\ref{prop:nearby-verdier}, \ref{prop:maxext}, and~\ref{prop:vancyc}) match those in~\cite{morel}.

\begin{lem}\label{lem:psi-formula}
For $\cF \in \Dmix_\Gm(\cX_\gen,\bk)$, there is a natural isomorphism $\Psi_f(\cF) \cong \Mon^{-1}\bi^!\bj_!\cJ(\cF)\la -2\ra[1]$.
\end{lem}
\begin{proof}
The functorial distinguished triangles $\bj_!\cJ(\cF) \to \bj_*\cJ(\cF) \to \bi_*\bi^*\bj_*\cJ(\cF) \to$ and $\bi_*\bi^!\bj_!\cJ(\cF) \to \bj_!\cJ(\cF) \to \bj_*\cJ(\cF) \to$ show that there is a natural isomorphism $\bi^!\bj_!\cJ(\cF)[1] \cong \bi^*\bj_*\cJ(\cF)$.  The result follows.
\end{proof}

\begin{lem}\label{lem:psi-sr-dt}
For $\cF \in \Dmix_\Gm(\cX_\gen,\bk)$, there are natural distinguished triangles
\begin{gather*}
\Psi_f(\cF) \xrightarrow{\Nilp} \Psi_f(\cF)\la 2\ra \to \bi^*\bj_*\For(\cF) \to, \\
\bi^!\bj_!\For(\cF) \to \Psi_f(\cF) \xrightarrow{\Nilp} \Psi_f(\cF)\la 2\ra \to.
\end{gather*}
\end{lem}
\begin{proof}
These triangles are obtained by applying $\Mon^{-1}\bi^*\bj_*$ and $\Mon^{-1}\bi^!\bj_!$, respectively, to the distinguished triangle in the third part of Proposition~\ref{prop:jordan}.
\end{proof}

\begin{prop}\label{prop:nearby-verdier}
For any $\cF \in \Dmix_\Gm(\cX_\gen,\bk)$, there is a natural isomorphism $\D\Psi_f(\cF) \cong \Psi_f(\D\cF)\la 2\ra$.
\end{prop}
\begin{proof}
Using Proposition~\ref{prop:jordan} and Lemma~\ref{lem:psi-formula}, we have
\begin{multline*}
\D\Psi_f(\cF) = \D(\Mon^{-1}\bi^*\bj_*\cJ(\cF)\la -2\ra) 
\cong \Mon^{-1}\bi^!\bj_!\D(\cJ(\cF))\la 2\ra \\
\cong \Mon^{-1}\bi^!\bj_!\cJ(\D\cF)[1] \cong \Psi_f(\D\cF)\la 2\ra,
\end{multline*}
as desired.
\end{proof}

\begin{hyp}\label{hyp:j-exact}
The functors $\bj_!, \bj_*: \Dmix_\Gm(\cX_\gen,\bk) \to \Dmix_\Gm(\cX,\bk)$ are $t$-exact for the perverse $t$-structure.
\end{hyp}

Since $\bj$ is the base change of the inclusion of the affine open subset $(\bA^1 \smallsetminus \{0\}) \hookrightarrow \bA^1$, it is an affine morphism. If we were working in the ordinary (nonmixed) derived category $\Dbc(\cX,\bk)$, Hypothesis~\ref{hyp:j-exact} would then follow from~\cite[Corollaire~4.1.3]{bbd}.  When $\bk$ is a field of characteristic~$0$ (and under some additional assumptions on $\cX$), Hypothesis~\ref{hyp:j-exact} can likely be proved using the techniques of~\cite{ar:kdsf}.

Hypothesis~\ref{hyp:j-exact} will be proved for some key examples in~\cite{arider}.  It remains open in general for now.

\begin{thm}
If Hypothesis~\ref{hyp:j-exact} holds, then the functor $\Psi_f: \Dmix_\Gm(\cX_\gen,\bk) \to \Dmix(\cX_0,\bk)$ is $t$-exact for the perverse $t$-structure.
\end{thm}
\begin{proof}
Let $\p\coh^i$ denote cohomology with respect to the perverse $t$-structure, and consider the map $\p\coh^i(\Nilp): \p\coh^i(\Psi_f(\cF)) \to \p\coh^i(\Psi_f(\cF)\la 2\ra)$.  If this map is either injective or surjective, the same is true for any power $\p\coh^i(\Nilp^n): \p\coh^i(\Psi_f(\cF)) \to \p\coh^i(\Psi_f(\cF)\la 2n\ra)$.  But by Remark~\ref{rmk:hom-nilp}, we have $\Nilp^n = 0$ for $n \gg 0$.  We conclude that if $\p\coh^i(\Nilp)$ is either injective or surjective, then in fact $\p\coh^i(\Psi_f(\cF)) = 0$.

Now let $\cF$ be a perverse sheaf in $\Dmix_\Gm(\cX_\gen,\bk)$, and consider the long exact perverse cohomology sequence for the distinguished triangle $\bj_!\cF \to \bj_*\cF \to \bi_*\bi^*\bj_*\cF \to$.  Hypothesis~\ref{hyp:j-exact} implies that $\p \coh^i(\bi^*\bj_*\cF)$ vanishes unless $i = -1,0$.  The long exact cohomology sequence associated to the distinguished triangle from Lemma~\ref{lem:psi-sr-dt} then shows that $\p\coh^i(\Nilp)$ is surjective for $i > 0$ and injective for $i < 0$.  Therefore, $\p\coh^i(\Psi_f(\cF)) = 0$ unless $i = 0$.
\end{proof}

\subsection{The maximal extension functor}

For $\cF \in \Perv_\Gm(\cX_\gen,\bk)$, let $\tilde\Xi_f(\cF)$ be the cone of the natural map $\bj_!\cJ(\cF)\la -2\ra \xrightarrow{\Nilp} \bj_*\cJ(\cF)$.  In other words, we have a distinguished triangle
\begin{equation}\label{eqn:txi-defn}
\bj_!\cJ(\cF)\la -2\ra \xrightarrow{\Nilp} \bj_*\cJ(\cF) \to \tilde\Xi_f(\cF) \to.
\end{equation}

\begin{lem}\label{lem:maxext-pre}
The assignment $\cF \mapsto \tilde\Xi_f(\cF)$ defines a functor $\Perv_\Gm(\cX_\gen,\bk) \to \Dmix_\mon(\cX,\bk)$.  Moreover, there are natural isomorphisms
\[
\bj^*\tilde\Xi_f(\cF) \cong \Mon(\For(\cF)),
\quad
\bi^*\tilde\Xi_f(\cF) \cong \bi^*\bj_*\cJ(\cF),
\quad
\bi^!\tilde\Xi_f(\cF) \cong \bi^!\bj_!\cJ(\cF)\la -2\ra[1].
\]
In particular, $\tilde\Xi_f(\cF)$ is constructible.
\end{lem}
\begin{proof}
Let $f: \cF \to \cG$ be a morphism in $\Perv_\Gm(\cX_\gen,\bk)$, and consider the diagram
\[
\begin{tikzcd}
\bj_!\cJ(\cF)\la -2\ra \ar[r, "\Nilp"] \ar[d, "\bj_!\cJ(f)\la -2\ra"'] &
\bj_*\cJ(\cF)\ar[r] \ar[d, "\bj_*\cJ(f)"'] &
\tilde\Xi_f(\cF) \ar[r] \ar[d, dashed, "h"] & {}
\\
\bj_!\cJ(\cG)\la -2\ra \ar[r, "\Nilp"] &
\bj_*\cJ(\cG)\ar[r] &
\tilde\Xi_f(\cG) \ar[r] & {}
\end{tikzcd}
\]
To show that $\tilde\Xi_f$ is a functor, it is enough to show that $h$ is unique. We will show this below. 

Applying $\bj^*$ to~\eqref{eqn:txi-defn}, we see that $\bj^*\tilde\Xi_f(\cG)$ is the cone of $\cJ(\cG)\la -2\ra \xrightarrow{\Nilp} \cJ(\cG)$.  By Proposition~\ref{prop:jordan}, we conclude that $\bj^*\tilde\Xi_f(\cG) \cong \Mon(\For(\cG))$.  The uniqueness of $h$ above follows from~\cite[Proposition~1.1.9]{bbd} and the following computation:
\begin{align*}
\Hom(\bj_!\cJ(\cF)\la -2\ra, \tilde\Xi_f(\cG)[-1])
&\cong \Hom(\cJ(\cF)\la -2\ra, \bj^*\tilde\Xi_f(\cG)[-1]) \\
&\cong \Hom(\cJ(\cF)\la -2\ra, \Mon(\For(\cG))[-1]) \\
&\cong \Hom(\Coi(\cJ(\cF))\la -2\ra, \cG[-1]) \\
&\cong \Hom(\cF\la -2\ra, \cG[-1]) = 0.
\end{align*}
Here, the last step follows from the fact that $\cF\la -2\ra$ and $\cG$ both lie in the heart of the perverse $t$-structure.

The descriptions of $\bi^*\tilde\Xi_f(\cF)$ and $\bi^!\tilde\Xi_f(\cF)$ are obtained by applying $\bi^*$ and $\bi^!$ to~\eqref{eqn:txi-defn}.  These objects are constructible by Theorem~\ref{thm:rfree-mon}.  By Lemma~\ref{lem:constr-recolle}, we conclude that $\tilde\Xi_f(\cF)$ is constructible.
\end{proof}

We define the \emph{maximal extension functor}
\[
\Xi_f: \Perv_\Gm(\cX_\gen,\bk) \to \Dmix(\cX,\bk)
\qquad\text{by}\qquad
\Xi_f(\cF) = \Mon^{-1}(\tilde\Xi_f(\cF)).
\]
By Lemma~\ref{lem:maxext-pre}, there is a natural isomorphism $\bj^*\Xi_f(\cF) \cong \For(\cF)$.

\begin{prop}\label{prop:maxext}
For $\cF \in \Perv_\Gm(\cX_\gen,\bk)$, there are  natural distinguished triangles
\begin{gather*}
\bj_!\cF \xrightarrow{\alpha_-} \Xi_f(\cF) \xrightarrow{\beta_-} \bi_*\Psi_f(\cF)\la 2\ra \to, \\
\bi_*\Psi_f(\cF) \xrightarrow{\beta_+} \Xi_f(\cF) \xrightarrow{\alpha_+} \bj_*\cF \to
\end{gather*}
such that $\alpha_+ \circ \alpha_-$ is the canonical map $\bj_!\cF \to \bj_*\cF$, and $\beta_- \circ \beta_+ = \Nilp: \Psi_f(\cF) \to \Psi_f(\cF)\la 2\ra$.

If Hypothesis~\ref{hyp:j-exact} holds, $\Xi_f$ is an exact functor $\Perv_\Gm(\cX_\gen,\bk) \to \Perv(\cX,\bk)$, and these are triangles are short exact sequences.
\end{prop}
\begin{proof}
Let $\hat \alpha_+$, $\hat\alpha_-$, $\hat\beta_+$, and $\hat\beta_-$ be the maps indicated in the natural distinguished triangles below:
\begin{gather*}
\bj_!\bj^*\tilde\Xi_f(\cF) \xrightarrow{\hat\alpha_-} \tilde\Xi_f(\cF) \xrightarrow{\hat\beta_-} \bi_*\bi^*\tilde\Xi_f(\cF) \to, \\
\bi_*\bi^!\tilde\Xi_f(\cF) \xrightarrow{\hat\beta_+} \tilde\Xi_f(\cF) \xrightarrow{\hat\alpha_+} \bj_*\bj^*\tilde\Xi_f(\cF) \to,
\end{gather*}
The distinguished triangles in statement of the proposition are obtained by applying $\Mon^{-1}$ to the triangles above, and then using the information from the end Lemma~\ref{lem:maxext-pre}.  It is clear from this construction that $\alpha_+ \circ \alpha_-$ is the natural map $\bj_!\cF \to \bj_*\cF$.

Next, we must show that $\beta_- \circ \beta_+ = \Nilp$.  Since $\beta_+$ is induced by the adjunction $\bi_*\bi^! \to \id$, our problem is equivalent to showing that $\bi^!\beta_-$ is the map induced by $\Nilp$.  To see this, we will study the following natural octahedral diagram:
\begin{equation}\label{eqn:max-oct1}
\hspace{-1em}
\begin{tikzcd}[column sep=tiny, row sep=small]
\bj_!\Mon(\cF) \ar[dd, "{\sst[1]}" description] && \bi_*\bi^*\bj_*\cJ(\cF) \ar[dl, "{\sst[1]}" description] \ar[ll, "{\sst[1]}" description] \\
& \bj_!\cJ(\cF) \ar[dr] \ar[ul] \\
\bj_!\cJ(\cF)\la -2\ra \ar[ur, "\Nilp"] \ar[rr] && \bj_*\cJ(\cF) \ar[uu]
\end{tikzcd}
\begin{tikzcd}[column sep=tiny, row sep=small]
\bj_!\Mon(\cF) \ar[dd, "{\sst[1]}" description] \ar[dr, "\hat\alpha_-"'] && \bi_*\bi^*\bj_*\cJ(\cF) \ar[ll, "{\sst[1]}" description] \\
& \tilde\Xi_f(\cF) \ar[dl, "{\sst[1]}" description] \ar[ur, "\hat\beta_-"'] \\
\bj_!\cJ(\cF)\la -2\ra \ar[rr] && \bj_*\cJ(\cF) \ar[uu] \ar[ul]
\end{tikzcd}
\hspace{-1em}
\end{equation}
Apply $\Mon^{-1}\bi^!$ to the diagram above. The resulting octahedron can be rewritten as shown below, using the definition of $\Psi_f$ along with Lemmas~\ref{lem:psi-formula} and~\ref{lem:maxext-pre}:
\begin{equation}\label{eqn:max-oct2}
\begin{tikzcd}[column sep=tiny, row sep=small]
\bi^!\bj_!\cF \ar[dd] && \Psi_f(\cF)\la 2\ra \ar[dl, "\id"'] \ar[ll, "{\sst[1]}" description] \\
& \Psi_f(\cF)\la 2\ra \ar[dr] \ar[ul, "{\sst[1]}" description] \\
\Psi_f(\cF) \ar[ur, "\Nilp"] \ar[rr] && 0 \ar[uu, "{\sst[1]}" description]
\end{tikzcd}
\quad
\begin{tikzcd}[column sep=tiny, row sep=small]
\bi^!\bj_!\cF \ar[dd] \ar[dr] && \Psi_f(\cF)\la 2\ra \ar[ll, "{\sst[1]}" description] \\
& \Psi_f(\cF) \ar[dl, "\id"'] \ar[ur, "\bi^!\beta_-"'] \\
\Psi_f(\cF) \ar[rr] && 0 \ar[uu, "{\sst[1]}" description] \ar[ul, "{\sst[1]}" description]
\end{tikzcd}
\end{equation}
Let us spell out these identifications we have made above in a bit more detail.  The left-hand distinguished triangle in the upper cap of~\eqref{eqn:max-oct2} comes from Lemma~\ref{lem:psi-sr-dt}.  To see that the map labelled ``$\id$'' in the upper cap is indeed the identity map, observe that the proof of Lemma~\ref{lem:psi-formula} involves the corresponding distinguished triangle in~\eqref{eqn:max-oct1}.  Similarly, the arrow labelled ``$\id$'' in the lower cap of~\eqref{eqn:max-oct2} comes from the proof of Lemma~\ref{lem:maxext-pre}.  The fact that $\bi^!\beta_- = \Nilp$ comes from the equality of the two paths from the center of the lower cap to the center of the upper cap in~\eqref{eqn:max-oct2}.

Finally, if Hypothesis~\ref{hyp:j-exact} holds, the last assertion in the proposition is clear.
\end{proof}

\begin{lem}\label{lem:xi-dual}
For $\cF \in \Perv_\Gm(\cX_\gen,\bk)$, there is a natural isomorphism $\D\Xi_f(\cF) \cong \Xi_f(\D\cF)$.  Moreover, $\D$ exchanges the two natural distinguished triangles in Proposition~\ref{prop:maxext}.
\end{lem}
\begin{proof}
By Proposition~\ref{prop:jordan}, applying $\D$ to~\eqref{eqn:txi-defn} yields the distinguished triangle
\[
\D\tilde\Xi_f(\cF) \to \bj_!\cJ(\D\cF)\la -2\ra[1] \xrightarrow{\Nilp} \bj_*\cJ(\D\cF)[1] \to.
\]
We therefore have $\D\tilde\Xi_f(\cF) \cong \tilde\Xi_f(\D\cF)$.  The reasoning in Lemma~\ref{lem:maxext-pre} shows that this isomorphism is natural.  The present lemma easily follows from this.
\end{proof}

\subsection{The vanishing cycles functor}

In this subsection, we assume Hypothesis~\ref{hyp:j-exact} throughout.

Recall from Proposition~\ref{prop:eqvt-ff} that the category $\Perv_\Gm(\cX_\gen,\bk)$ can be identified with a full subcategory of $\Perv(\cX_\gen,\bk)$.  Consider the category of ``generically $\Gm$-equivariant perverse sheaves,'' defined by
\[
\Perv_{\Gm,\gen}(\cX,\bk) = \{ \cF \in \Perv(\cX,\bk) \mid \text{$\bj^*\cF$ lies in $\Perv_\Gm(\cX_\gen,\bk)$} \}.
\]
This is a full abelian subcategory of $\Perv(\cX,\bk)$ (but it is not, in general, a Serre subcategory).

Given $\cF \in \Perv_{\Gm,\gen}(\cX,\bk)$, consider the natural maps
\[
\gamma_-: \bj_!(\cF|_{\cX_\gen}) \to \cF
\qquad\text{and}\qquad
\gamma_+: \cF \to \bj_*(\cF|_{\cX_\gen}).
\]
Consider also the maps $\alpha_-: \bj_!(\cF|_{\cX_\gen}) \to \Xi_f(\cF|_{\cX_\gen})$ and $\alpha_+: \Xi_f(\cF|_{\cX_\gen}) \to  \bj_*(\cF|_{\cX_\gen})$ from Proposition~\ref{prop:maxext}.  These maps satisfy $\alpha_+ \circ \alpha_- = \gamma_+ \circ \gamma_-$.  Therefore, the diagram
\[
\tilde\Phi_f(\cF) = \left(
\bj_!(\cF|_{\cX_\gen})
\xrightarrow{[\begin{smallmatrix} \alpha_- \\ \gamma_- \end{smallmatrix}]}
(\Xi_f(\cF|_{\cX_\gen}) \oplus \cF)
\xrightarrow{[\begin{smallmatrix} \alpha_+ & -\gamma_+ \end{smallmatrix}]}
\bj_*(\cF|_{\cX_\gen})
\right)
\]
can be regarded as a chain complex of perverse sheaves concentrated in degrees $-1$, $0$, and $1$.  The terms in this chain complex are all exact functors of $\cF$, so the assignment $\cF \mapsto \tilde\Phi_f(\cF)$ is an exact functor from $\Perv_{\Gm,\gen}(\cX,\bk)$ to the abelian category of chain complexes over $\Perv_{\Gm,\gen}(\cX,\bk)$.   

Moreover, since $\alpha_-$ is injective and $\alpha_+$ is surjective, the chain complex $\tilde\Phi_f(\cF)$ has cohomology only in degree $0$.  The functor $\cF \mapsto \coh^0(\tilde\Phi_f(\cF))$ is an exact functor $\Perv_{\Gm,\gen}(\cX,\bk) \to \Perv_{\Gm,\gen}(\cX,\bk)$.

Finally, note that $\bj^*\tilde\Phi_f(\cF)$ can be identified with
\[
\cF|_{\cX_\gen} 
\xrightarrow{[\begin{smallmatrix} \id \\ \id  \end{smallmatrix}]}
(\cF|_{\cX_\gen} \oplus \cF|_{\cX_\gen})
\xrightarrow{[\begin{smallmatrix} \id & -\id \end{smallmatrix}]}
\cF|_{\cX_\gen}
\]
This chain complex is acyclic.  In other words, $\bj^*\coh^0(\tilde\Phi_f(\cF)) = 0$, so $\tilde\Phi_f(\cF)$ is supported on $\cX_0$.

We define the \emph{vanishing cycles functor}
\[
\Phi_f: \Perv_{\Gm,\gen}(\cX,\bk) \to \Perv(\cX_0,\bk)
\qquad\text{by}\qquad
\Phi_f(\cF) = \bi^* \coh^0(\tilde\Phi_f(\cF)).
\]

Next, consider the following diagram:
\begin{equation}\label{eqn:canvar-defn}
\begin{tikzcd}[column sep=large,ampersand replacement=\&]
\& \bj_*(\cF|_{\cX_\gen}) \\
\bi_*\Psi_f(\cF|_{\cX_\gen}) \ar[r, "{[\begin{smallmatrix} \beta_+ \\ 0 \end{smallmatrix}]}"] \&
  \Xi_f(\cF|_{\cX_\gen}) \oplus \cF
    \ar[r, "{[\begin{smallmatrix} \beta_- & 0 \end{smallmatrix}]}"] 
    \ar[u, "{[\begin{smallmatrix} \alpha_+ & -\gamma_+ \end{smallmatrix}]}"]\&
\bi_*\Psi_f(\cF|_{\cX_\gen})\la 2\ra
\\
\& \bj_!(\cF|_{\cX_\gen}) \ar[u, "{[\begin{smallmatrix} \alpha_- \\ \gamma_- \end{smallmatrix}]}"]
\end{tikzcd}
\end{equation}
If we regard each column as a chain complex of perverse sheaves (with the first and third columns concentrated in degree $0$), then the horizontal maps are chain maps. Let us take their induced maps in cohomology: we define
\begin{align*}
\can&: \Psi_f(\cF|_{\cX_\gen}) \to \Phi_f(\cF),
&
\can &= \bi^*\coh^0([\begin{smallmatrix} \beta_+ \\ 0 \end{smallmatrix}]), \\
\var&: \Phi_f(\cF) \to \Psi_f(\cF|_{\cX_\gen})\la 2\ra,
&
\var &= \bi^*\coh^0([\begin{smallmatrix} \beta_- & 0 \end{smallmatrix}])
\end{align*}

\begin{prop}\label{prop:vancyc}
Assume that Hypothesis~\ref{hyp:j-exact} holds.  The vanishing cycles functor $\Phi_f: \Perv_{\Gm,\gen}(\cX,\bk) \to \Perv(\cX_0,\bk)$ is exact, and the natural transformations $\can: \Psi_f \to \Phi_f$ and $\var: \Phi_f \to \Psi_f\la 2\ra$ satisfy $\var \circ \can = \Nilp$.  Moreover, for any $\cF \in \Perv_{\Gm,\gen}(\cX,\bk)$ there are natural distinguished triangles
\begin{gather*}
\Psi_f(\cF|_{\cX_\gen}) \xrightarrow{\can} \Phi_f(\cF) \to \bi^*\cF \to, \\
\bi^!\cF \to \Phi_f(\cF) \xrightarrow{\var} \Psi_f(\cF|_{\cX_\gen})\la 2\ra \to.
\end{gather*}
\end{prop}
\begin{proof}
The fact that $\var \circ \can = \Nilp$ follows from the fact that $\beta_- \circ \beta_+ = \Nilp$ (see Proposition~\ref{prop:maxext}).  Next, consider the following diagram, in which the columns are chain complexes of perverse sheaves, and the horizontal maps are chain maps:
\begin{equation}\label{eqn:can-dt}
\begin{tikzcd}[column sep=large,ampersand replacement=\&]
\bj_*(\cF|_{\cX_\gen}) \ar[r, "\id"] \& \bj_*(\cF|_{\cX_\gen}) \\
\Xi_f(\cF|_{\cX_\gen}) \ar[r, "{[\begin{smallmatrix} \id \\ 0 \end{smallmatrix}]}"]
  \ar[u, "\alpha_+"] \&
  \Xi_f(\cF|_{\cX_\gen}) \oplus \cF
    \ar[r, "{[\begin{smallmatrix} 0 & \id \end{smallmatrix}]}"] 
    \ar[u, "{[\begin{smallmatrix} \alpha_+ & -\gamma_+ \end{smallmatrix}]}"]\&
\cF
\\
\& \bj_!(\cF|_{\cX_\gen}) \ar[u, "{[\begin{smallmatrix} \alpha_- \\ \gamma_- \end{smallmatrix}]}"] \ar[r, "\id"] \&
  \bj_!(\cF|_{\cX_\gen}) \ar[u, "\gamma_-"] 
\end{tikzcd}
\end{equation}
This is a natural short exact sequence of chain complexes, so it determines a natural distinguished triangle in the derived category $\Db\Perv(\cX,\bk)$.  Via a ``realization'' functor in the sense of~\cite[\S3.1]{bbd} or~\cite{bei:dcps}, it determines a natural distinguished triangle in $\Dmix(\cX,\bk)$. 

Moreover, by Proposition~\ref{prop:maxext}, the first column of~\eqref{eqn:can-dt} is canonically isomorphic to $\bi_*\Psi_f(\cF|_{\cX_\gen})$.  The last column is canonically isomorphic to $\bi_*\bi^*\cF$.  This diagram thus gives rise to the first distinguished triangle in the statement of the proposition.  The construction of the second distinguished triangle is very similar.
\end{proof}

\begin{lem}\label{lem:phi-dual}
For any object $\cF \in \Perv_{\Gm,\gen}(\cX,\bk)$, there is a natural isomorphism $\D\Phi_f(\cF) \cong \Phi_f(\D\cF)$.  Moreover, $\D$ exchanges the two natural distinguished triangles in Proposition~\ref{prop:vancyc}.
\end{lem}
\begin{proof}
By Lemma~\ref{lem:xi-dual}, the chain complex $\tilde\Phi_f(\cF)$ is self-dual under Verdier duality, so $\D\Phi_f(\cF) \cong \Phi_f(\D\cF)$.  Then, one checks that the dual of the diagram~\eqref{eqn:can-dt} is naturally isomorphic to the diagram used to define the second distinguished triangle in Proposition~\ref{prop:vancyc}.
\end{proof}

\section{Examples}
\label{sec:examples}

We conclude the paper with a couple of examples. See~\cite{arider} for full details and further generalizations.

\subsection{The identity map}

As in Example~\ref{ex:gm-der}, let $X = \bA^1$, stratified as the union of $X_0 = \{0\}$ and $X_1 = \bA^1 \smallsetminus \{0\}$.  Let $f: \cX \to \bA^1$ be the identity map.  In this case, the nearby cycles functor
\[
\Psi_f: \Dmix_\Gm(\bA^1 \smallsetminus \{0\}, \bk) \to \Dmix(\pt,\bk)
\]
is an equivalence of categories.  It is $t$-exact, but it does not preserve weights: we have
\[
\Psi_f(\ubk\{1\}) \cong \ubk_\pt\la -1\ra.
\]
To see this, one can first check that $\bj_!\cJ(\ubk\{1\})$ is the object described in Example~\ref{ex:mon-der1}.  It follows that
\[
\bi^!\bj_!\cJ(\ubk\{1\})\la -2\ra[1] \cong 
\left(\begin{tikzcd}[column sep=large]
\ubk_\pt\la -1\ra \ar[r, shift left=1.5, "{\sst[1]}" description, "\xi\cdot \id" above=2] &
\ubk_\pt\la -3\ra[1] \ar[l, shift left=1.5, "{\sst[1]}" description, "\sr\cdot \id" below=2]
\end{tikzcd}\right)
\cong \Mon(\ubk_\pt\la -1\ra).
\]
For any $\cF \in \Dmix_\Gm(\bA^1 \smallsetminus \{0\},\bk)$, the endomorphism $\Nilp: \Psi_f(\cF) \to \Psi_f(\cF)\la 2\ra$ is zero.

The maximal extension of the constant sheaf is given by
\[
\Xi_f(\ubk_{\bA^1 \smallsetminus \{0\}}\{1\}) =
\begin{tikzcd}
\ubk_\pt\la 1\ra \ar[r, "{\sst[1]}" description, "\eta" below=2]
  \ar[rr, bend left=20, "{\sst[1]}" description, "-\bxi \cdot \id" above=2] &
  \ubk_{\bA^1}\{1\} \ar[r, "{\sst[1]}" description, "\epsilon" below=2] &
  \ubk_\pt\la -1\ra
\end{tikzcd}
\]
where $\eta$ and $\epsilon$ are the adjunction maps.  This object is a ``tilting perverse sheaf'' in the sense of~\cite{bbm}.  See~\cite[Remark~1.1(ii)]{bbm}, and see~\cite[Example~4.6.4]{amrw1} for a related object in the context of flag varieties.

\subsection{The product map on \texorpdfstring{$\bA^2$}{A2}}

Let $X = \bA^2$.  Let $H = \Gm$, and let $H$ act on $X$ by $z\cdot(x_1,x_2) = (zx_1, z^{-1}x_2)$.  We also take a second copy of $\Gm$ and let it act on $X$ by $z \cdot(x_1,x_2) = (x_1,zx_2)$.  Stratify $X$ by the orbits of the $\Gm \times H$-action.  There are four strata, which we denote by
\begin{gather*}
\pt = \{(0,0)\},
\qquad
X_1 = (\bA^1 \smallsetminus \{0\}) \times \{0\},
\qquad
X_2 = \{0\} \times (\bA^1 \smallsetminus \{0\}),
\\
U = (\bA^1 \smallsetminus \{0\}) \times (\bA^1 \smallsetminus \{0\}).
\end{gather*}
Next, let $f: X \to \bA^1$ be the map $f(x_1,x_2) = x_1x_2$.  This map is $\Gm \times H$-equivariant, where $H$ acts trivially on $\bA^1$.  Note that $X_\gen = U$, and $X_0 = f^{-1}(0)$ is the union of the other three strata.  Again, $\Psi_f$ is $t$-exact.  It can be shown that
\[
\Psi_f(\ubk\{2\}) \cong
\begin{tikzcd}[column sep=large, ampersand replacement=\&]
\ubk_\pt \ar[r, "{\sst[1]}" description, "{[\begin{smallmatrix} \eta_1 \\ -\eta_2\end{smallmatrix}]}" below=2] 
  \ar[rr, bend left=20, "{\sst[1]}" description, "-\bxi \cdot \id" above=2] \&
(\ubk_{X_1}\{1\} \oplus \ubk_{X_2}\{1\})\la -1\ra
\ar[r, "{\sst[1]}" description, "{[\begin{smallmatrix} \epsilon_1 & -\epsilon_2 \end{smallmatrix}]}" below=2] \&
\ubk_\pt\la -2\ra
\end{tikzcd}
\]
Here, the $\eta_i$ and $\epsilon_i$ are appropriate unit and counit maps.  They satisfy $\epsilon_1\eta_1 + \epsilon_2\eta_2 = \xi \in \Hom(\ubk_\pt, \ubk_\pt\{2\}) = \coh^2_{\Gm \times H}(\pt,\bk)$.  The monodromy endomorphism can be read off this picture using Proposition~\ref{prop:constr-monodromy}.  This object is closely related to the object described in~\cite[\S1.2.3]{gaitsgory}.  See~\cite{arider} for further details.


\end{document}